\documentclass{article}

\usepackage{authblk}

\usepackage[utf8]{inputenc}
\usepackage[T1]{fontenc}
\usepackage{textcomp}
\usepackage{amsmath,amssymb,amsthm}
\usepackage{lmodern}
\usepackage[a4paper]{geometry}
\usepackage{xcolor, pict2e}
\usepackage{microtype}
\usepackage{caption}
\usepackage{listings}
\usepackage{multicol}
\usepackage{moreverb}
\usepackage{hyperref}
\hypersetup{pdfstartview=XYZ}
\usepackage{wrapfig}
\usepackage[sans]{dsfont}

\usepackage{hyperref}
\hypersetup{
    colorlinks=true,
    linkcolor=blue,
    citecolor=red,
    urlcolor=blue,
    pdfborder={0 0 0}
}

\usepackage[font=sf, labelfont={sf,bf}, margin=1cm]{caption}

\usepackage{cite}

\newtheorem*{theorem*}{Theorem}
\newtheorem{theorem}{Theorem}[subsection]
\newtheorem{proposition}{Proposition}[subsection]
\newtheorem{lemma}{Lemma}[subsection]

\newtheorem{question}{Question}

\newtheorem{theoremm}{Theorem}

\theoremstyle{remark}
\newtheorem{remark}{Remark}[subsection]
\newtheorem{assumption}{Assumption}

\usepackage{tikz}

\title{Thick points of random walk and the Gaussian free field}
\author{Antoine Jego\thanks{Supported by the EPSRC grant EP/L016516/1 for the University of Cambridge Centre for Doctoral Training, the Cambridge Centre for Analysis. E-mail address: \href{mailto:apfj2@cam.ac.uk}{apfj2@cam.ac.uk}}
}
\affil{University of Cambridge}
\date {}

\newcommand{\R}{\mathbb{R}}
\newcommand{\N}{\mathbb{N}}
\newcommand{\Z}{\mathbb{Z}}
\newcommand{\C}{\mathbb{C}}
\newcommand{\Bc}{\mathcal{B}}
\newcommand{\Cc}{\mathcal{C}}
\newcommand{\Mc}{\mathcal{M}}
\newcommand{\Nc}{\mathcal{N}}

\newcommand{\expect}{\mathbb{E}}
\newcommand{\Expect}[1]{\mathbb{E} \left[ #1 \right] }
\newcommand{\EXPECT}[2]{\mathbb{E}_{#1} \left[ #2 \right] }
\newcommand{\prob}{\mathbb{P}}
\newcommand{\Prob}[1]{\mathbb{P} \left( #1 \right) }
\newcommand{\PROB}[2]{\mathbb{P}_{#1} \left( #2 \right) }

\newcommand{\EXPECTprod}[2]{\mathbb{E}_{#1} \otimes \mathds{E} \left[ #2 \right] }
\newcommand{\PROBprod}[2]{\mathbb{P}_{#1} \otimes \mathds{P} \left( #2 \right) }

\newcommand{\abs}[1]{\left\vert #1 \right\vert}
\newcommand{\norme}[1]{\left\| #1 \right\| }
\newcommand{\scalar}[1]{\left\langle #1 \right\rangle }
\newcommand{\floor}[1]{\left\lfloor #1 \right\rfloor}
\newcommand{\indic}[1]{ \mathbf{1}_{ \left\{ #1 \right\} } }
\newcommand{\eps}{\varepsilon}

\begin{document}

\renewcommand{\theparagraph}{\thesubsection.\arabic{paragraph}} 

\maketitle

\begin{abstract}
We consider the thick points of random walk, i.e. points where the local time is a fraction of the maximum. In two dimensions, we answer a question of \cite{dembo2001} and compute the number of thick points of planar random walk, assuming that the increments are symmetric and have a finite moment of order two. The proof provides a streamlined argument based on the connection to the Gaussian free field and works in a very general setting including isoradial graphs. In higher dimensions, we study the scaling limit of the set of thick points. In particular, we show that the rescaled number of thick points converges to a nondegenerate random variable and that the centred maximum of the local times converges to a randomly shifted Gumbel distribution.
\end{abstract}


\tableofcontents

\section{Results}\label{sec: Introduction and Results}

For $d \geq 2$, consider a continuous time simple random walk $(Y_t)_{t \geq 0}$ on $\Z^d$ with rate 1. Let us denote $\prob_x$ the law of $(Y_t)_{t \geq 0}$ starting from $x$ and $\expect_x$ the associated expectation. Defining $V_N = \{-N, \dots, N\}^d$, we denote $\tau_N$ the first exit time of $V_N$ and $\left(\ell_x^t, x \in \Z^d, t \geq 0 \right)$ the local times defined by:
\begin{equation}\label{eq: def tau_N and local times}
\tau_N := \inf \left\{ t \geq 0, Y_t \notin V_N \right\} \mathrm{~and~} \forall x \in V_N, \forall t\geq 0, 
\ell_x^t := \int_0^{t} \indic{Y_s = x} ds.
\end{equation}

In 1960, Erd\H{o}s and Taylor \cite{erdos_taylor1960} studied the behaviour of the local time of the most frequently visited site. By translating their work to our context of continuous time random walk, they proved that
\begin{align}
\mathrm{if~} d=2,&~ \frac{1}{\pi} \leq \liminf_{N \rightarrow \infty} \frac{\sup_{x \in V_N} \ell_x^{\tau_N}}{(\log N)^2} \leq \limsup_{N \rightarrow \infty} \frac{\sup_{x \in V_N} \ell_x^{\tau_N}}{(\log N)^2} \leq \frac{4}{\pi} \quad \quad \prob_0 \mathrm{-a.s.}, \nonumber \\
\mathrm{if~} d \geq 3,&~ \lim_{N \rightarrow \infty} \frac{\sup_{x \in V_N} \ell_x^{\tau_N}}{\log N} = 2 \EXPECT{0}{ \ell_0^\infty } \quad \quad \prob_0 \mathrm{-a.s.} \label{eq: d=3 erdos-taylor result}
\end{align}
and conjectured that the limit also exists in dimension two and is equal to the upper bound. This conjecture was proved forty years later in a landmark paper \cite{dembo2001}. Estimates on the number of thick points, which are the points where the local times are larger than a fraction of the maximum, are also given in this paper. Briefly, their proof establishes the analogous results for the thick points of occupation measure of planar Brownian motion; taking in particular advantages of symmetries such as rotational invariance and certain exact computations on Brownian excursions. The discrete case is then deduced from the Brownian case through strong coupling/KMT arguments. This method requires all the moments of the increments to be bounded but the authors suspected that only finite second moments are needed. Later, the article \cite{rosen2006} showed that the paper \cite{dembo2001} can be entirely rewritten in terms of random walk giving a proof without using Brownian motion. The strategy of \cite{rosen2006} has then been refined in \cite{BassRosen2007} to treat the case of random walks on $\Z^2$ with symmetric increments having finite moment of order $3+\eps$. A crucial aspect of this latter article consists in controlling the jumps over discs. Such a control is achieved by developing Harnack inequalities requiring further assumptions on the walk (Condition A of \cite{BassRosen2007}).

This paper has two purposes. Firstly, we exploit the links between the local times and the Gaussian free field (GFF) provided by Dynkin-type isomorphisms to give a simpler and more robust proof of the two-dimensional result. The proof works in a very general setting (Theorem \ref{th: frequent points, general case}). In particular, we weaken the assumptions of \cite{BassRosen2007} answering the question of \cite{dembo2001} about walks with only finite second moments and we also treat the case of random walks on isoradial graphs.
Secondly, we obtain more precise results in dimension $d \geq 3$. Namely, we show that the field $\left\{ \ell_x^{\tau_N}, x \in V_N \right\}$ behaves like the field composed of i.i.d. exponential variables with mean $\EXPECT{0}{\ell_0^\infty}$ located at each site visited by the walk. In particular, we show that the centred supremum of the local times as well as the rescaled number of thick points converge to nondegenerate random variables.

We first state two results for the planar case. Both are in fact corollaries of a more general theorem (Theorem \ref{th: frequent points, general case}) which will be stated later. We will then present the result in dimension $d \geq 3$.

\subsection{Dimension two}\label{subsec: dim 2}

Consider $Y_t = S_{N_t}, t \geq 0,$ a continuous time random walk on $\Z^2$ starting from the origin where $S_n = \sum_{i = 1}^n X_i, n \geq 0,$ is the jump process with i.i.d. increments $X_i \in \Z^2$ and $(N_t)_{t \geq 0}$ is an independent Poisson process of parameter $1$. As before, we consider the square $V_N$ of side length $2N+1$, the first exit time $\tau_N$ of $V_N$ and the local times $\left( \ell_x^t,x \in \Z^2, t \geq 0 \right)$ defined as in \eqref{eq: def tau_N and local times}. For any thickness parameter $0 \leq a \leq 1$, we call $\Mc_N(a)$ the set of $a$-thick points
\[
\Mc_N(a) := \left\{ x \in V_N: \ell_x^{\tau_N} \geq \frac{2}{\pi \sqrt{\det \mathcal{G}}}a ( \log N)^2  \right\}
\]
where $\mathcal{G}$ is defined below.
Then we have the following:

\begin{theorem}\label{thm: lattices}
Assume that the law of the increments is symmetric (i.e. $-X \stackrel{d}{=} X$), with a finite variance and denote $\mathcal{G} = \Expect{XX'}$ the covariance matrix of the increments. Then we have the following two a.s. limits:
\[
\lim_{N \rightarrow \infty} \frac{\max_{x \in V_N} \ell_x^{\tau_N}}{(\log N)^2} = \frac{2}{\pi \sqrt{\det \mathcal{G}}}
\mathrm{~and~}
\forall a \in [0,1),
\lim_{N \rightarrow \infty} \frac{\log \abs{\Mc_N(a)}}{\log N} = 2(1 - a).
\]
\end{theorem}

This theorem answers a question asked in the last section of \cite{dembo2001} with the additional assumption of symmetry. The assumption of symmetry is needed in our approach since otherwise we cannot define an associated GFF.

\bigbreak
Our approach is sufficiently general that it can handle random walks with a very different flavour; for instance we discuss here the case of random walk on isoradial graphs.

We recall briefly the definitions and introduce some notation (we use the same one as \cite{Chelkak_Smirnov2011}).
Let $\Gamma = (V,E)$ be any connected infinite isoradial graph, with common radius 1, i.e. $\Gamma$ is embedded in $\C$ and each face is inscribed into a circle of radius 1.
Note that if $x, y \in V$ are adjacent then $x$ and $y$, together with the centres of the two faces adjacent to the edge $\{x,y\}$, form a rhombus. We denote by $2 \theta_{x,y}$ the interior angle of this rhombus at $x$ (or at $y$). See Figure \ref{fig: rhombic half-angle} for an example. For instance, the square (resp. triangular, hexagonal, etc) lattice is an isoradial graph with $\theta_{x,y} = \pi/4$ (resp. $\pi / 6, \pi / 3$, etc) for all $x \sim y$. We assume the following ellipticity condition:
\[
\exists \eta \in \left( 0, \frac{\pi}{4} \right), \forall x \sim y, \theta_{x,y} \in \left( \eta, \frac{\pi}{2} - \eta \right).
\]

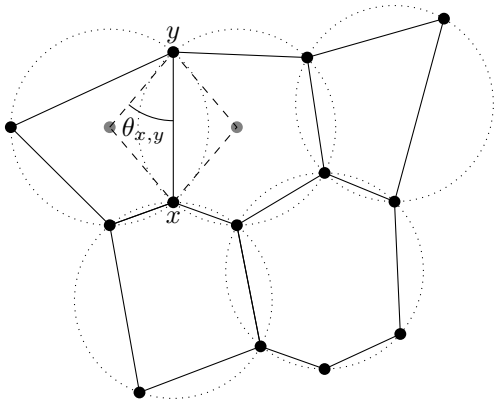
\begin{figure}
\centering
\begin{tikzpicture}[scale=1.3]
\coordinate (A) at (1,1);
\coordinate (B) at (0.357,2.766);
\coordinate (C) at (1.643,2.766);
\coordinate (D) at (2.53,1.3);
\tikzset{shift={(A)}}
\coordinate (A1) at (-28:1cm);
\coordinate (A2) at (50:1cm);
\coordinate (A3) at (90:1cm);
\coordinate (A4) at (130:1cm);
\coordinate (A5) at (250:1cm);
\tikzset{shift={(B)}}
\coordinate (B1) at (180:1cm);
\coordinate (B2) at (50:1cm);
\tikzset{shift={(D)}}
\coordinate (D1) at (90:1cm);
\coordinate (D2) at (45:1cm);
\coordinate (D3) at (270:1cm);
\coordinate (D4) at (320:1cm);
\tikzset{shift={(D1)}}
\coordinate (E) at (45:1cm);
\tikzset{shift={(E)}}
\coordinate (E1) at (152:1cm);
\coordinate (G1) at (60:1cm);

\fill (G1) circle (0.06cm);
\fill (D3) circle (0.06cm);
\fill (D4) circle (0.06cm);
\fill (A5) circle (0.06cm);

\fill (A1) circle (0.06cm);
\fill (A2) circle (0.06cm);
\fill (A3) circle (0.06cm);
\fill (A4) circle (0.06cm);
\fill (B1) circle (0.06cm);
\fill (B2) circle (0.06cm);
\fill (D1) circle (0.06cm);
\fill (D2) circle (0.06cm);
\fill (E1) circle (0.06cm);
\fill[gray] (B) circle (0.06cm);
\fill[gray] (C) circle (0.06cm);

\draw (A1) -- (A2) -- (A3) -- (A4) -- (A5) -- (A1);
\draw (B1) -- (B2) -- (A3) -- (A4) -- (B1);
\draw (D1) -- (A2) -- (A1) -- (D3) -- (D4) -- (D2) -- (D1);
\draw (B2) -- (E1) -- (D1);
\draw (D2) -- (G1) -- (E1);

\draw[dashed] (B) -- (B2) -- (C) -- (A3) -- (B);

\draw[dotted] (A) circle (1cm);
\draw[dotted] (B) circle (1cm);
\draw[dotted] (C) circle (1cm);
\draw[dotted] (D) circle (1cm);
\draw[dotted] (E) circle (1cm);

\draw (A3) node [below] {$x$};
\draw (B2) node [above] {$y$};
\tikzset{shift={(B2)}}
\draw (0,-0.7) arc (-90:-130:0.7);
\draw (-0.3,-0.8) node {$\theta_{x,y}$};

\end{tikzpicture}
\caption{Isoradial graph and rhombic half-angle. The solid lines represent the edges of the graph. Each face is inscribed into a dotted circle of radius 1.
The centres of the two faces adjacent to the edge $\{x,y\}$ are in grey.}\label{fig: rhombic half-angle}
\end{figure}

Define $\forall x \sim y \in V$ the conductance $c_{x,y} = \tan(\theta_{x,y})$ and let $(Y_t)_{t \geq 0}$ be a Markov jump process with conductances $(c_e)_{e \in E}$. $Y$ is a continuous time walk which waits an exponential with mean $1/ \sum_{y \sim x} c_{x,y}$ time in each vertex $x$ and then jumps from $x$ to $y$ with probability $c_{x,y} / \sum_{z \sim x} c_{x,z}$. Take a starting point $x_0 \in V$ and denoting $d_\Gamma$ the graph distance we define for all $N \in \N,$
\[
V_N := \{x \in V: d_\Gamma(x,x_0) \leq N \}
\]
and as before (equation \eqref{eq: def tau_N and local times}), we consider the first exit time $\tau_N$ of $V_N$ and the local times.
We will denote $\prob_x$ the law of the walk $(Y_t)_{t \geq 0}$ starting from $x \in V$ and $\expect_x$ the associated expectation.

As confirmed by the theorem below, a sensible definition of $a$-thick points is given by
\[
\Mc_N(a) := \left\{ x \in V_N: \ell_x^{\tau_N} \geq \frac{a}{\pi} ( \log N)^2  \right\}.
\]

\begin{theorem}\label{th: frequent points, isoradial graphs}
We have the following two $\prob_{x_0}$-a.s. limits:
\[
\lim_{N \rightarrow \infty} \frac{\max_{x \in V_N} \ell_x^{\tau_N}}{(\log N)^2} = \frac{1}{\pi}
\mathrm{~and~}
\forall a \in [0,1),
\lim_{N \rightarrow \infty} \frac{\log \abs{\Mc_N(a)}}{\log N} = 2(1 - a).
\]
\end{theorem}

\begin{remark} Theorems \ref{thm: lattices} and \ref{th: frequent points, isoradial graphs} also hold when we consider the walk stopped at a deterministic time, $N^2$ say, rather than the first exit time $\tau_N$ of $V_N$, since
\[
\lim_{N \rightarrow \infty} \frac{\log \tau_N}{\log N} = 2 \quad \quad \mathrm{a.s.}
\]
(easy to check but can also be seen from these two theorems). They also hold if we consider discrete time random walks rather than continuous time random walks. In that case, we have to multiply the discrete local times by the average time the continuous time walk stays in a given vertex before its first jump. See Remark \ref{rem:discrete time versus continuous time} ending Section \ref{subsec: dim 3} for a short discussion about this.

Let us just confirm that Theorems \ref{thm: lattices} and \ref{th: frequent points, isoradial graphs} are coherent: in the square lattice case, the average time between successive jumps by the walk $Y$ of Theorem \ref{th: frequent points, isoradial graphs} is $1/4$ rather than $1$.
We also mention that it is plausible that the arguments of \cite{rosen2006} can be adapted to show Theorem \ref{th: frequent points, isoradial graphs}. However, we include it here since it is a straightforward consequence of our approach (Theorem \ref{th: frequent points, general case}).
\end{remark}

\subsection{Higher dimensions}\label{subsec: dim 3}

We now come back to the setting of the beginning of Section \ref{sec: Introduction and Results} for $d \geq 3$ and we denote $g := \EXPECT{0}{\ell_0^\infty}$. In this section, the walk starts at the origin of $\Z^d$.

We describe thick points through a more precise encoding by considering for $a \in [0,1]$ the point measure:
\begin{equation}\label{eq: dim 3 def nu^a_N}
\nu^a_N := \frac{1}{N^{2(1-a)}} \sum_{x \in V_N} \delta_{ \left( x/N , \ell_x^{\tau_N} - 2g a \log N \right) }.
\end{equation}
Let us emphasise that the normalisation factor is equal to 1 when $a = 1$ and that
$\nu_N^a$ is viewed as a random measure on $[-1,1]^d \times \R$. We compare the thick points of random walk with the thick points of i.i.d. exponential random variables with mean $g$ located at each site visited by the walk. More precisely, we denote $\Mc_N(0) := \left\{ x \in V_N: \ell_x^{\tau_N} > 0 \right\}$ and taking $E_x, x \in \Z^d,$ i.i.d. exponential variables with mean $g$ independent of $\Mc_N(0)$, we define
\[
\mu^a_N := \frac{1}{N^{2(1-a)}} \sum_{x \in \Mc_N(0)} \delta_{ \left( x/N , E_x - 2g a \log N \right) }.
\]

We finally denote by $\tau$ the first exit time of $[-1,1]^d$ of Brownian motion starting at the origin and by
$\mu_{\mathrm{occ}}$ the occupation measure of Brownian motion starting at the origin and killed at $\tau$.
Then we have:

\begin{theorem}\label{th: dim 3 convergence measures}
For all $a \in [0,1]$ there exists a random Borel measure $\nu^a$ on $[-1,1]^d \times \R$ such that, with respect to the topology of vague convergence of measures on $[-1,1]^d \times \R$ (on $[-1,1]^d \times (0,\infty)$ if $a=0$), we have:
\[
\lim_{N \rightarrow \infty} \nu^a_N = \lim_{N \rightarrow \infty} \mu^a_N = \nu^a \mathrm{~in~law.}
\]
Moreover, for all $a \in [0,1)$ the distribution of $\nu^a$ does not depend on $a$ and
\begin{equation}\label{eq:thm measure subcritical}
\nu^a(dx,d\ell) \overset{\mathrm{(d)}}{=} \frac{1}{g} \mu_{\mathrm{occ}}(dx) \otimes e^{-\ell/g} \frac{d \ell}{g}.
\end{equation}
At criticality, $\nu^1$ is a Poisson point process:
\begin{equation}\label{eq:thm measure critical}
\nu^1 \overset{\mathrm{(d)}}{=} \mathrm{PPP} \left( \frac{1}{g} \mu_{\mathrm{occ}}(dx) \otimes e^{-\ell/g} \frac{d \ell}{g} \right).
\end{equation}
\end{theorem}

We will see that this statement will imply the following two theorems:

\begin{theorem}\label{th: dim 3 convergence thick points}
If we define for every $a \in [0,1]$ the set of $a$-thick points:
\[
\Mc_N(a) := \left\{ x \in V_N: \ell_x^{\tau_N} > 2ga \log N \right\},
\]
then there exist random variables $M_a$ such that for all $a \in [0,1]$
\[
\frac{\abs{\Mc_N(a)}}{N^{2(1-a)}} \xrightarrow[N \rightarrow \infty]{\mathrm{(d)}}
M_a
.
\]
Moreover, for all $a \in [0,1)$ the distribution of $M_a$ does not depend on $a$ and
\begin{equation}\label{eq:thm law number thick points subcritical}
M_a \overset{\mathrm{(d)}}{=} \tau / g.
\end{equation}
$M_1$ is a Poisson variable with parameter $\tau / g$: for all $k \geq 0$
\begin{equation}\label{eq:thm law number thick points critical}
\Prob{M_1 = k} = \frac{1}{k!} \Expect{ e^{-\frac{\tau}{g}}  \left( \frac{\tau}{g} \right)^k }.
\end{equation}
\end{theorem}

\begin{theorem}\label{th: dim 3 convergence supremum}
There exists an almost surely finite random variable $L$ such that
\[
\sup_{x \in V_N} \ell_x^{\tau_N} - 2g \log N \xrightarrow[N \rightarrow \infty]{\mathrm{(d)}} L.
\]
Moreover, $L$ is a Gumbel variable with mode $g \log (\tau/g)$ (location of the maximum) and scale parameter $g$, i.e. for all $t \in \R$
\[
\Prob{L \leq t} = \Expect{ \exp \left( - \frac{\tau}{g} e^{-t/g} \right) }.
\]
\end{theorem}

To the best of our knowledge, this result is not present in the current literature. A detailed study of the local times of random walk in dimension greater than two has been done in a series of papers by Cs\'aki, F\"{o}ldes, R\'{e}v\'{e}sz, Rosen and Shi (see \cite{CFR_2007} for a survey of this work). 
In particular, Theorem 1 of \cite{Revesz_2004} and the corollary following the main theorem of \cite{CFR_2006} improved the estimate of Erd\H{o}s and Taylor (equation \eqref{eq: d=3 erdos-taylor result}). By translating their work to our setting of continuous time random walk (see the next remark), they showed that a.s. for all $\eps > 0$, there exists $N_0 < \infty$ a.s. such that for all $N \geq N_0$,
\[
-(4+\eps)g \log \log N \leq \sup_{x \in V_N} \ell_x^{\tau_N} - 2 g \log N \leq (2+\eps) g \log \log N.
\]
Let us also mention the fact that Theorem 2 of \cite{Revesz_2004} states that for all $\eps>0$, almost surely we have $\sup_{x \in V_N} \ell_x^{\tau_N} - 2g \log N \geq ( 2(d-4)/(d-2) - \eps ) \log \log N$ for infinitely many $N$. This is not in contradiction with our Theorem \ref{th: dim 3 convergence supremum} because we only give the typical behaviour (i.e. at a fixed time) of $\sup_{x \in V_N} \ell_x^{\tau_N} - 2g \log N$.

\begin{remark}\label{rem:discrete time versus continuous time}
We have stated our results in the case of continuous time random walk but they hold as well for discrete time random walk. As already mentioned, the statements in the planar case do not need to be changed. The reason for this is because in dimension two we were essentially comparing exponential (continuous time) or geometrical (discrete time) variables with mean $g \log N$ to $a g (\log N)^2$ for some $g >0$ and $a \in (0,1)$. In both cases, if we divide these variables by $g \log N$ then they converge to exponential variables with parameter 1. Thus there is no difference between the continuous time case and the discrete time one. On the contrary, in higher dimensions, we are comparing exponential or geometrical variables with mean $g$ to $g a \log N$ and these two distributions have slightly different behaviour. In the discrete time setting, our results claim that the field composed of the local times behaves like the field composed of independent geometrical variables with mean $g$ located at each site visited by the walk. Theorems \ref{th: dim 3 convergence measures}--\ref{th: dim 3 convergence supremum} then have to be modified accordingly.
\end{remark}

\section{Outline of proofs and literature overview}\label{sec: Literature Overview and Organization of the Paper}

Section \ref{sec:Dimension 2} will be dedicated to the dimension two whereas Section \ref{sec:Dimension 3} will deal with the dimensions greater or equal to three. Let us first describe the two dimensional case.

We first recall the definition of the GFF on the square lattice. With the notations of Theorem \ref{th: frequent points, isoradial graphs} in the square lattice case, the Gaussian free field is the centred Gaussian field $\phi_N$, indexed by the vertices in $V_N$, whose covariances are given by the Green function:
\[
\mathds{E}[\phi_N(x) \phi_N(y)] = \EXPECT{x}{\ell_y^{\tau_N}}.
\]
See \cite{berestycki}, \cite{zeitouni_notes} for introductions to the GFF.
Our argument will simply relate the thick points of the random walk to those of the GFF: see \cite{kahane}, \cite{HuMillerPeres2010} in the continuum and \cite{bolthausen2001}, \cite{daviaud2006} in the discrete case.

We now explain the interest of exploiting the connection to the GFF. As usual, the proofs of Theorems \ref{thm: lattices} and \ref{th: frequent points, isoradial graphs} rely on the method of (truncated) second moment. That is, a first moment estimate on $\abs{\Mc_N(a)}$ gives us the upper bound, while a matching upper bound on the second moment of $\abs{\Mc_N(a)}$ would supply the lower bound. Moreover, it is necessary to first consider a truncated version of $\abs{\Mc_N(a)}$, where we consider points that are never too thick at all scales (this is similar to the idea in \cite{berestycki2017}). Computing the corresponding correlations is not easy with the random walk, but is essentially straightforward with the GFF as this is basically part of the definition. As only an upper bound on the second moment is needed, comparisons to the GFF with Dynkin-type isomorphisms go in the right direction. We will see that the Eisenbaum's version will be the most convenient to work with.

We now state this isomorphism. Consider $\Gamma = (V,E)$ a non-oriented connected infinite graph without loops, not necessary planar, equipped with symmetric conductances $(W_{xy})_{x,y \in V}$. Let $E'$ be the edge set $E' = \{ \{x,y\} : x,y \in V, W_{xy} >0 \}$. Let $\prob_x$ be the law under which $(Y_t)_{t \geq 0}$ is a symmetric Markov jump process with conductances $(W_{yz})_{y,z \in V}$ (i.e. jump rates $W_{yz}$ from $y$ to $z$) starting at $x$ at time 0. $Y$ is thus a nearest neighbour random walk on $(V,E')$ but not necessary on $\Gamma = (V,E)$.
As in the isoradial case, we denote $\ell_x^t, x \in V, t \geq 0$, its local times, $x_0$ a starting point, $V_N$ the ball of radius $N$ and centre $x_0$ for the graph distance of $\Gamma$, $\tau_N$ the first exit time of $V_N$.
Because $Y$ is a \textit{symmetric} Markov process, the following expression is symmetric in $x,y$:
\[
\EXPECT{x}{\ell_y^{\tau_N}} = \EXPECT{y}{\ell_x^{\tau_N}}.
\]
This allows us to define a centred Gaussian field $\phi_N$ whose covariances are given by the previous expression. $\phi_N$ is called Gaussian free field and we will denote $\mathds{P}$ its law.
The following theorem establishes a relation between the local times and the GFF (see lectures notes \cite{rosen_2014} for a good overview of this topic)

\begin{theoremm}[Eisenbaum's isomorphism]\label{th: Eisenbaum isomorphism}
For all $s>0$ and all measurable bounded function $f : \R^{V_N} \rightarrow \R$,
\begin{align*}
&\EXPECTprod{x_0}{ f \left\{ \left( \ell_x^{\tau_N} + \frac{1}{2} (\phi_N(x) + s)^2 \right)_{x \in V_N} \right\} } \\
& =
\mathds{E} \left[ \left(1 + \frac{\phi_N(x_0)}{s} \right) f \left\{ \left( \frac{1}{2} (\phi_N(x) + s)^2 \right)_{x \in V_N} \right\} \right].
\end{align*}
\end{theoremm}

\begin{remark}\label{rem: eisenbaum's isomorphism}
We are now going to explain why we chose to use this isomorphism instead of the maybe more well-known generalised second Ray-Knight theorem. To ease the comparison, we are going to state this other isomorphism in the setting that is of interest to us. 
Consider the graph $(V_N,E_N)$ with $E_N = \{ \{x,y\} : x,y \in V_N, W_{xy} > 0 \}$. Let $\prob_x$ be the law under which $(Y_t)_{t \geq 0}$ is a symmetric Markov jump process with conductances $(W_e)_{e \in E_N}$ starting at $x$ at time 0. Let $\ell_x^t,x \in V_N, t>0$, be the associated local times and for $u>0$, define 
$\tau_u := \inf \{ t>0: \ell_{x_0}^t \geq u \}$ and $\tau_{x_0} := \inf \{t >0: Y_t = x_0 \}$. We can now define $\mathds{P}$ the law under which $(\psi_N(x), x \in V_N)$ is the GFF in $V_N$ with zero-boundary condition at $x_0$, i.e. $\psi_N$ is a centred Gaussian vector whose covariance matrix is given by
\[
\mathds{E} [\psi_N(x) \psi_N(y)] = \EXPECT{x}{\ell_y^{\tau_{x_0}}}.
\]
The generalised second Ray-Knight theorem states that (see again the lecture notes \cite{rosen_2014}):
\begin{equation}
\label{eq:secondRayKnight}
\left( \ell_x^{\tau_u} + \frac{1}{2} \psi_N(x)^2 \right)_{x \in V_N}
\overset{\mathrm{(d)}}{=}
\left( \frac{1}{2} \left( \psi_N(x) + \sqrt{2u} \right)^2 \right)_{x \in V_N}
\end{equation}
under $\prob_{x_0} \otimes \mathds{P}$ and $\mathds{P}$.

It would have been possible to use this isomorphism to show Theorems \ref{thm: lattices} and \ref{th: frequent points, isoradial graphs}. Compared to the Eisenbaum's isomorphism above, this has the advantage that the laws of the GFFs on the left hand side and right hand side are the same. However this has a drawback: indeed it is necessary to stop the walk where it starts, i.e. at $x_0$. This isomorphism then leads to a GFF $\psi_N$ pinned at $x_0$. This is essentially equivalent to adding a global noise to the Dirichlet GFF $\phi_N$ of order $\sqrt{\log N}$ which is sufficient to ruin second moment approach. This noise would have to be removed by hand in order to apply the method of second moment. This is possible but makes the proof substantially longer.

The generalised second Ray-Knight isomorphism has been used several times to study problems related to local times (see for instance \cite{CoverBlanketTimes}). We now mention two works that are maybe the most relevant to us.
The isomorphism \eqref{eq:secondRayKnight} immediatly gives the following stochastic domination:
\[
\left( \sqrt{\ell_x^{\tau_u}} \right)_{x \in V_N} \prec \left( \frac{1}{\sqrt{2}} \abs{ \psi_N(x) + \sqrt{2u} } \right)_{x \in V_N}
\]
under $\prob_{x_0}$ and $\mathds{P}$.
One can actually show a stronger result and replace the absolute value on the right hand side by $\max( \cdot ,0)$ (Theorem 3.1 of \cite{Zhai2014}). Abe \cite{abe2015RayKnight} exploited this and used the symmetry of the GFF to make links between what was called thin points and thick points of the random walk on the two-dimensional torus, up to a multiple of the cover time.

Let us also mention that Abe and Biskup \cite{AbeBiskup} have announced a work in preparation which relates the thick points of random walk to the Liouville quantum gravity in dimension two. This is in the same spirit as this paper as they also rely on a connection to the GFF. However, we emphasise some important differences. First, the walk they consider is on a box and has wired boundary conditions, meaning that the walk is effectively re-randomised every time it hits the boundary of the box. Second, they consider the local time profile at a regime comparable to the cover time, so that the comparison to the GFF is perhaps more clear.
\end{remark}

\paragraph*{Organisation - planar case:}

The two-dimensional part of the paper will be organised as follows. In Section \ref{sec: General Framework and Upper Bound} we will present the general framework that we treat (Theorem \ref{th: frequent points, general case}). We will then show that Theorems \ref{thm: lattices} and \ref{th: frequent points, isoradial graphs} are simple corollaries. The upper bound, which is the easy part, will be briefly proved at the end of the same section.
Section \ref{sec: Lower Bound} is devoted to the lower bound.
We first show that the probability to have a lot of thick points does not decay too quickly. This is the heart of our proof and makes use of the comparison to the GFF.
We then bootstrap this argument to obtain the same statement with high probability, see Lemma \ref{lem: from not too small proba to proba 1} at the beginning of Section \ref{sec: Lower Bound}.
This lemma is a key feature of our proof and allows us to use the comparison to the GFF. Indeed, since we do not require very precise estimates, we can deal with the change of measure coming from the isomorphism through very rough bounds, such as: $\abs{\phi_N(x_0)} \leq (\log N)^2$ with high probability (see Lemma \ref{lem: upper bound event Eisenbaum}). This only introduces a poly-logarithmic multiplicative error in the estimate of the probabilities that two given points are thick, and so does not matter for the computation of the fractal dimension of the number of thick points on a polynomial scale.

If we want more accurate estimates, more ideas are required. For instance, for the simple random walk on the square lattice, the comparison between the number of thick points for the random walk and for the GFF breaks down: the two following expectations converge as $N$ goes to infinity:
\begin{gather}
\label{eq: approximation first moment thick points RW}
\lim_{N \rightarrow \infty}
\frac{\log N}{N^{2(1-a)}} \EXPECT{0}{ \# \left\{ x \in V_N: \ell_x^{\tau_N} \geq \frac{4a}{\pi} ( \log N)^2 \right\} }
\in (0, \infty),
\\
\label{eq: approximation first moment thick point GFF}
\lim_{N \rightarrow \infty}
\frac{\sqrt{\log N}}{N^{2(1-a)}} \mathds{E} \left[ \# \left\{ x \in V_N: \tfrac{1}{2} \phi_N(x)^2 \geq \frac{4a}{\pi} ( \log N)^2 \right\} \right]
\in (0, \infty).
\end{gather}
In the article \cite{BiskupLouidor} the thick points of the discrete GFF $\phi_N$ were encoded in point measures of a similar form as the one we defined in \eqref{eq: dim 3 def nu^a_N}. The authors showed the convergence of such measures. As a consequence, they went beyond the estimate \eqref{eq: approximation first moment thick point GFF} and showed that
\begin{equation}\label{eq:LouidorBiskup}
\frac{\sqrt{\log N}}{N^{2(1-a)}} \# \left\{ x \in V_N: \tfrac{1}{2} \phi_N(x)^2 \geq \frac{4a}{\pi} ( \log N)^2 \right\}
\end{equation}
converges in law to a nondegenerate random variable.

\begin{question}
In the case of simple random walk on the square lattice starting at the origin, does
\begin{equation}\label{eq:myquestion}
\frac{\log N}{N^{2(1-a)}} \# \left\{ x \in V_N: \ell_x^{\tau_N} \geq \frac{4a}{\pi} ( \log N)^2 \right\}
\end{equation}
converge to a nondegenerate random variable as $N$ goes to infinity?
\end{question}

Notice that the renormalisations are different in \eqref{eq:LouidorBiskup} and in \eqref{eq:myquestion}. These differences suggest scraping the GFF approach if we want optimal estimates. This is what we will do in higher dimensions.

\textit{
Update: after this work was completed, this question has been solved in \cite{jegoGMC}, \cite{AbeBiskup} and \cite{jegoRW}. 
The framework of \cite{jegoRW} is the above-described setting of planar random walk stopped upon hitting the boundary of $V_N$ for the first time, whereas \cite{jegoGMC} works in an analogue setting for planar Brownian motion. 
The article \cite{AbeBiskup} considers different type of walks that are run up to a time proportional to the cover time of a planar graph and that have wired boundary condition (see Remark \ref{rem: eisenbaum's isomorphism}).
}

\bigbreak

We have finished to discuss the two-dimensional case and we now describe the situation in higher dimensions. The article \cite{DPRZ_2000_spatial} studied the thick points of occupation measure of Brownian motion in dimensions greater or equal to three. They obtained the leading order of the maximum and computed the Hausdorff dimension of the set of thick points. The article \cite{CFRRS_2005}, as well as \cite{CFR_2005}, \cite{CFR_2006}, \cite{CFRbis_2007}, \cite{CFRter_2007} (again, see \cite{CFR_2007} for a survey on this series of paper), studied the case of symmetric transient random walk on $\Z^d$ with finite variance.
One of their results computed the leading order of the maximum of the local times too. In both \cite{DPRZ_2000_spatial} and \cite{CFRRS_2005}, a key feature of the proofs is a localisation property (Lemma 3.1 of \cite{DPRZ_2000_spatial} and Lemma 2.2 of \cite{CFRRS_2005}) which roughly states that a thick point accumulates most of its local time in a short interval of time. This property allows them to consider independent variables and makes the situation simpler compared to the two-dimensional case.

Let us also mention the paper \cite{chiarini2015} which studied the scaling limit of the discrete GFF in dimension greater or equal to three. The authors obtained a result similar to Theorem \ref{th: dim 3 convergence measures}. Namely, they showed that in the limit the field behaves as independent Gaussian variables. More precisely, they defined a point process analogue to $\nu_N^1$ (see \eqref{eq: dim 3 def nu^a_N}) which encodes the thickest points of the GFF. They showed that this point process converges to a Poisson point process. Their situation is simpler because the intensity measure is governed by the Lebesgue measure rather than the occupation measure of Brownian motion. In particular, they could use the Stein-Chen method which allowed them to consider only the two first moments.

\paragraph*{Organisation - higher dimensions:}
Let us now present the main lines of our proofs and the organisation of the paper.
In Section \ref{sec: dim 3 proofs theorems}, Theorems \ref{th: dim 3 convergence measures}, \ref{th: dim 3 convergence thick points} and \ref{th: dim 3 convergence supremum} will all be obtained from the joint convergence of the sequences of real-valued random variables $\nu^a_N(A_1 \times T_1), \dots, \nu^a_N(A_r \times T_r)$, for all suitable $A_i \subset [-1,1]^d$ and $T_i \subset \R$. We will obtain this fact by computing explicitly all the moments of these variables (Proposition \ref{prop: dim 3 moment number thick points}). This is actually the heart of our proofs and Section \ref{sec: dim 3 proof proposition} will be entirely dedicated to it.
To compute the $k$-th moment of $\nu^a_N(A \times T)$, we will estimate the probability that the local times in $k$ different points, say $x_1, \dots, x_k$, belong to $2ga \log N + T$. In the subcritical regime ($a < 1$), we will be able to assume that these points are far away from each other. In that case, Lemma \ref{lem: dim 3 number of excursions E} will show that we can restrict ourselves to the event that there exists a permutation $\sigma$ of the set of indices $\{1, \dots, k\}$ which orders the vertices so that we have the following: the walk first hits $x_{\sigma(1)}$, accumulates a big local time in $x_{\sigma(1)}$, then hits $x_{\sigma(2)}$, accumulates a big local time in $x_{\sigma(2)}$, etc. When the walk has visited $x_{\sigma(i)}$ it does not come back to the vertices $x_{\sigma(1)}, \dots, x_{\sigma(i-1)}$. The local times can thus be treated as if they were independent.

At criticality ($a=1$), we do not renormalise the number of thick points and we will a priori have to take into account points which are close to each other. Here, the key observation - contained in Lemma \ref{lem: dim 3 two close points are not both thick} and already present in Corollary 1.3 of \cite{CFRRS_2005} - is that if two distinct points are close to each other, then the probability that they are both thick is much smaller than the probability that one of them is thick, even if they are neighbours! This is specific to the dimension greater or equal to 3 and tells that the thick points do not cluster. Thus, only the points which are either equal or far away from each other will contribute to the $k$-th moment.

Section \ref{sec: dim 3 proof lemmas} will contain the proofs of four intermediate lemmas that are needed to prove Proposition \ref{prop: dim 3 moment number thick points} on the convergence of the moments of $\nu^a_N(A_1 \times T_1), \dots, \nu^a_N(A_r \times T_r)$ for suitable $A_i \subset [-1,1]^d$ and $T_i \subset \R$.

\section{Dimension two}\label{sec:Dimension 2}

\subsection{General framework and upper bound}\label{sec: General Framework and Upper Bound}

We now describe the general setup for the theorem. Consider $\Gamma = (V,E)$ a non-oriented connected infinite graph without loops, not necessary planar, equipped with \textit{symmetric} conductances $(W_{xy})_{x,y \in V}$.
As before, we take $x_0 \in V$ a starting point and write $d_\Gamma$ for the graph distance. We will also write
\[
\forall N \in \N, V_N(x_0) := \{x \in V: d_\Gamma(x,x_0) \leq N \}.
\]
Let $\prob_x$ be the law under which $(Y_t)_{t \geq 0}$ is a symmetric Markov jump process with conductances $(W_{yz})_{y,z \in V}$ (i.e. jump rates $W_{yz}$ from $y$ to $z$) starting at $x$ at time 0. $Y$ is thus a nearest neighbour random walk on $(V,E')$, where $E' = \{ \{x,y\} : x,y \in V, W_{xy} >0 \}$, but not necessary on $\Gamma$.
We introduce the first exit time of $V_N(x_0)$ and the local times:
\[
\tau_N(x_0) := \inf \left\{ t \geq 0, Y_t \notin V_N(x_0) \right\} \mathrm{~and~} \forall x \in V, \forall t\geq 0, 
\ell_x^t := \int_0^{t} \indic{Y_s = x} ds.
\]
Finally we will denote $G_N^{x_0}$ the Green function, i.e.:
\begin{equation}\label{eq: definition Green function}
G_N^{x_0}(x,y) := \EXPECT{x}{\ell_y^{\tau_N(x_0)}}.
\end{equation}
If there is no confusion, we will simply write $V_N, \tau_N$ and $G_N$ instead of $V_N(x_0)$, $\tau_N(x_0)$ and $G_N^{x_0}$.

\textbf{Notation:} For two real-valued sequences $(u_N)_{N \geq 1}$ and $(v_N)_{N \geq 1}$ and for some parameter $\alpha$, we will denote $u_N = o_\alpha(v_N)$ if
\begin{align*}
\forall \eps>0, \exists N_0 = N_0(\alpha, \eps) >0, \forall N \geq N_0, \abs{u_N} \leq \eps \abs{v_N},
\end{align*}
and we will denote $u_N = O_\alpha(v_N)$ if
\begin{align*}
\exists C = C(\alpha) > 0, \exists N_0 = N_0(\alpha), \forall N \geq N_0, \abs{u_N} \leq C \abs{v_N}.
\end{align*}

We now make the following assumptions on the graph $\Gamma$ and on the walk $Y$:

\paragraph{Assumptions}\label{para: Assumptions}

We start with two assumptions on the geometry of the graph $\Gamma$.

\begin{assumption}\label{ass1}
$\# V_N(x_0) = N^{2+o(1)}$ and for all $x'_0 \in V_N(x_0)$
there exists a subset $Q_N(x'_0) \subset V_N(x'_0)$ with $N^{2+o(1)}$ points such that
\begin{equation}
\forall \alpha < 2, \sum_{x,y \in Q_N(x'_0)} \left( \frac{N}{d_\Gamma(x,y) \vee 1} \right)^\alpha = N^{4+o_\alpha(1)}. \label{eq: hypothesis on Q_N 2}
\end{equation}
\end{assumption}

\begin{assumption}\label{ass3}
For all $\eta \in (0,1)$, $x_0' \in V_N(x_0)$, $x \in Q_N(x'_0)$ and $R \in [1, N^{1-\eta}]$, we can find a subset $C_R(x) \subset Q_N(x'_0)$ which can be thought of as a circle of radius $R$ centred at $x$:
\begin{subequations} \label{eq: hypothesis existence circles}
\begin{align}
\forall y \in C_R(x), \log \frac{R}{d_\Gamma(x,y)} & = o_\eta(\log N), \label{eq: hypothesis existence circles, distance} \\
\frac{1}{\# C_R(x)^2} \sum_{y,y' \in C_R(x)} \log \left( \frac{R}{d_\Gamma(y,y') \vee 1 } \right) & = o_\eta( \log N ). \label{eq: hypothesis existence circles, control sum}
\end{align}
\end{subequations}
\end{assumption}

We now assume that we have good controls on the Green function:

\begin{assumption}\label{ass2}
There exists $g > 0$ such that:
\begin{subequations}\label{eq: hypothesis on the Green functions}
\begin{align}
\forall x \in V_N(x_0), G_N^{x_0}(x,x) & \leq g \log N + o(\log N), \label{eq: hypothesis upper bound G_N} \\
\forall x_0' \in V_N(x_0), \forall x, y \in Q_N(x'_0), G_N^{x_0'}(x,y) & = g \log \left( \frac{N}{d_\Gamma(x,y) \vee 1} \right) + o(\log N), \label{eq: hypothesis G_N on Q_N} \\
\forall x_0' \in V_N(x_0), \forall x \in Q_N(x'_0), G_N^{x_0'}(x'_0,x) & \geq (1/N)^{o(1)}. \label{eq: hypothesis G_N starting point}
\end{align}
\end{subequations}
\end{assumption}

Finally, we assume that the jumps are not unreasonable:

\begin{assumption}\label{ass4}
For all $K_N = N^{1-o(1)} \leq N$, $x_0' \in V_{N-K_N}(x_0)$ and $M >0$,
\begin{equation}\label{eq: hypothesis jumps}
\PROB{x_0'}{ d_\Gamma \left( x_0', Y_{\tau_{K_N}(x_0')} \right) \geq K_N+M} \leq K_N N^{o(1)} / M.
\end{equation}
where $\tau_{K_N}(x_0')$ is the first exit time of $V_{K_N}(x_0')$.
\end{assumption}

We now briefly discuss the above assumptions.
Note that we have assumed that all the bounds do not depend on the starting point $x'_0 \in V_N(x_0)$. This will be important for our Lemma \ref{lem: from not too small proba to proba 1}. Assumption \ref{ass3} is needed to go beyond the $L^2$ phase whereas Assumption \ref{ass4} is needed to bootstrap the probability to have a lot of thick points (Lemma \ref{lem: from not too small proba to proba 1}). This latter assumption can be weakened. We could replace $K_N N^{o(1)}/M$ by $f(K_N N^{o(1)}/M)$ with a function $t \in (0, \infty) \mapsto f(t) \in (0,\infty)$ which goes to zero quickly enough as $t$ goes to zero. For instance, any positive power of $t$ would do.

As confirmed by the theorem below, a sensible definition of $a$-thick points is given by
\[
\Mc_N(a) := \left\{ x \in V_N: \ell_x^{\tau_N} \geq 2ag ( \log N)^2  \right\}.
\]

\begin{theorem}\label{th: frequent points, general case}
Assuming the above assumptions we have the following two $\prob_{x_0}$-a.s. convergences:
\[
\lim_{N \rightarrow \infty} \frac{\max_{x \in V_N} \ell_x^{\tau_N}}{(\log N)^2} = 2g
\mathrm{~and~}
\forall a \in [0,1),
\lim_{N \rightarrow \infty} \frac{\log \abs{\Mc_N(a)}}{\log N} = 2(1 - a).
\]
\end{theorem}

We now check that Theorems \ref{thm: lattices} and \ref{th: frequent points, isoradial graphs} are consequences of this last theorem. Theorem \ref{th: frequent points, isoradial graphs} naturally fits into the setting of continuous time random walks defined using symmetric conductances, whereas the setting of Theorem \ref{thm: lattices} corresponds to the above-described general framework with $\Gamma$ being the square lattice equipped with weights $W_{xy} = \Prob{X = y-x}$. These weights are symmetric thanks to the assumption $X \overset{\mathrm{(d)}}{=} - X$. 
We now need to check that these two setups satisfy Assumptions \ref{ass1} - \ref{ass4} above.

For the isoradial case, the walk is a nearest-neighbour random walk so Assumption \ref{ass4} is clear. The following lemma checks that all the other assumptions are fulfilled if we define
\[
\forall x'_0 \in V_N(x_0), Q_N(x'_0) := \left\{
\begin{array}{cc}
V_{N/R_N}(x'_0) & \mathrm{~in~the~square~lattice~case}, \\
V_{\eps N}(x'_0) & \mathrm{~in~the~isoradial~case},
\end{array}
\right. \\
\]
where $R_N$ and $\eps$ are defined Lemma \ref{lem:check} below,
and if we define in both cases
\begin{align*}
\forall x'_0 \in V_N(x_0), \forall x \in Q_N(x'_0), \forall R \geq 1, C_R(x) & := \{ y \in Q_N(x'_0): d_\Gamma(x,y) = R \}.
\end{align*}

\begin{lemma}\label{lem:check}
\begin{enumerate}
\item Square Lattice.
Consider a walk $Y$ as in Theorem \ref{thm: lattices} and denote by $\mathcal{G}$ the covariance matrix of the increments. Let $x_0' \in \Z^2$ be a starting point. Then there exists $C>0$ independent of $x_0'$ such that for all $M>0$,
\begin{equation}\label{eq: first statement lemma}
\PROB{x_0'}{ d_\Gamma \left( x_0', Y_{\tau_N(x_0')} \right) \geq N+M} \leq C N / M.
\end{equation}
Moreover for all $\eta \in (0,1)$,
\begin{align}
\forall x,y \in V_N(x_0'), G_N^{x_0'}(x,y) & \leq \frac{1}{\pi \sqrt{\det \mathcal{G}}} \log \left( \frac{N}{\abs{x-y} \vee 1} \right) + o( \log N ), \label{eq: second statement lemma1} \\
\forall x,y \in V_{(1-\eta)N}(x_0'), G_N^{x_0'}(x,y) & \geq \frac{1}{\pi \sqrt{\det \mathcal{G}}} \log \left( \frac{N}{\abs{x-y} \vee 1} \right) + o_\eta( \log N ) \label{eq: second statement lemma2}
\end{align}
and there exists a sequence $R_N = N^{o(1)}$ such that
\begin{equation}
\label{eq:lem_16c1}
\forall x \in V_{N/R_N}(x_0'), G_N^{x_0'}(0,x_0') \geq N^{o(1)}.
\end{equation}

\item Isoradial Graphs.
Consider a walk $Y$ as in Theorem \ref{th: frequent points, isoradial graphs}. Let $x_0' \in V$ be a starting point. Then for all $\eta \in (0,1)$,
\begin{align}
\forall x,y \in V_N(x_0'), G_N^{x_0'}(x,y) & \leq \frac{1}{2 \pi} \log \left( \frac{N}{\abs{x-y} \vee 1} \right) + C, \label{eq: statement lemma isoradial graphs 1} \\
\forall x,y \in V_{(1-\eta)N}(x_0'), G_N^{x_0'}(x,y) & \geq \frac{1}{2 \pi} \log \left( \frac{N}{\abs{x-y} \vee 1} \right) - C(\eta) \label{eq: statement lemma isoradial graphs 2}
\end{align}
for some $C, C(\eta) >0$ independent of $x_0'$. Moreover, there exist $c, \eps>0$ independent of $x_0'$ such that 
\begin{equation}\label{eq:lem_16c2}
\forall x \in V_{\eps N}(x_0'), G_N^{x_0'}(x_0',x) \geq c.
\end{equation}
\end{enumerate}
\end{lemma}

\begin{proof}
\emph{Square lattice.}
We first start to prove \eqref{eq: first statement lemma}. By translation invariance, we can assume that $x_0' = 0$. We consider the discrete time random $(S_i)_{i \geq 0}$ associated and we are going to abusively write $\tau_N$ to denote the first time the discrete time walk exits $V_N$.
Take $\lambda >0$ to be chosen later on. The probability we are interested in is not larger than
\begin{align*}
\PROB{0}{ d_\Gamma \left( S_{\tau_N-1}, S_{\tau_N} \right) \geq M }
& \leq \PROB{0}{ \exists i \leq \tau_N -1, d_\Gamma (S_i, S_{i+1}) \geq M } \\
& \leq \PROB{0}{ \exists i \leq \lambda N^2 -1, d_\Gamma (S_i, S_{i+1}) \geq M } + \PROB{0}{\tau_N > \lambda N^2}.
\end{align*}
As the increments have a finite variance, the first term on the right hand side is not larger than $C \lambda N^2 / M^2$ for some $C>0$ by the union bound. Secondly,
\[
\PROB{0}{\tau_N > \lambda N^2} \leq \PROB{0}{d_\Gamma\left(0,S_{\lambda N^2} \right) \leq N}.
\]
Theorem 2.3.9 of \cite{lawler_limic_2010} gives estimates on the heat kernel and in particular implies that there exists $C>0$ such that for all $x \in \Z^2$, $\PROB{0}{S_i = x} \leq C/i$. Hence
\[
\PROB{0}{\tau_N > \lambda N^2} \leq C' / \lambda.
\]
We obtain \eqref{eq: first statement lemma} by taking $\lambda = M/N$.

Now, \eqref{eq: second statement lemma1} and \eqref{eq: second statement lemma2} are consequences of the estimate on the potential kernel $a(x)$ made in Theorem 4.4.6 of \cite{lawler_limic_2010}:
\[ a(x) = \frac{1}{\pi \sqrt{\det \mathcal{G}}} \log \abs{x} + o(\log \abs{x}) \mathrm{~as~} \abs{x} \rightarrow \infty \]
which is linked to the Green function by:
\begin{equation}\label{eq:green and potential kernel}
G_N(x,y) = \sum_{z \in V_N^c} \PROB{x}{Y_{\tau_N} = z} a(y-z) - a(y-x).
\end{equation}
If $z \in V_N^c$ is such that $d_\Gamma(x_0,z) \leq N (\log N)^2$, then
\[
\frac{1}{\pi \sqrt{\det \mathcal{G}}} \log N + o_\eta (\log N) \leq  a(y-z) \leq \frac{1}{\pi \sqrt{\det \mathcal{G}}} \log N + o (\log N)
\]
where the lower bound (resp. upper bound) is satisfied by all $y \in V_{(1-\eta) N}$ (resp. $V_N$). \eqref{eq: first statement lemma} implying that $\PROB{x}{d_\Gamma \left( x_0, Y_{\tau_N} \right) \leq N (\log N)^2} = 1 + o(1)$, we are thus left to show that the elements $z$ such that $d_\Gamma(x_0, z) > N (\log N)^2$ do not contribute to the sum in the equation \eqref{eq:green and potential kernel}. Thanks to \eqref{eq: first statement lemma}, we have
\begin{align*}
& \sum_{\substack{z \in \Z^2 \\ d_\Gamma(x_0,z) > N (\log N)^2}} \PROB{x}{Y_{\tau_N} = z} \log \abs{z} \\
& \leq \sum_{p = 0}^\infty \PROB{x}{ 2^p \leq d_\Gamma \left( x_0, Y_{\tau_N} \right) / (N (\log N)^2) < 2^{p+1} } \log \left( N (\log N)^2 2^{p+1} \right) \\
& \leq \frac{C}{(\log N)^2} \sum_{p=0}^\infty \frac{1}{2^p} \log \left( N (\log N)^2 2^{p+1} \right)
\leq \frac{C'}{\log N}
\end{align*}
which goes to zero as $N$ goes to infinity. It completes the proof of \eqref{eq: second statement lemma1} and \eqref{eq: second statement lemma2}. \eqref{eq:lem_16c1} is a direct consequence of \eqref{eq: second statement lemma2}.

\smallbreak
\emph{Isoradial graphs.}
\eqref{eq: statement lemma isoradial graphs 1} and \eqref{eq: statement lemma isoradial graphs 2} are a direct consequences of Theorem 1.6.2 and Proposition 1.6.3 of \cite{lawler1996intersections} in the case of simple random walk on the square lattice. Kenyon extended this result to general isoradial graphs (see \cite{Kenyon2002} or Theorem 2.5 and Definition 2.6 of\cite{Chelkak_Smirnov2011}). \eqref{eq:lem_16c2} follows from \eqref{eq: statement lemma isoradial graphs 2}.
\end{proof}

From now on, we will work with a graph $\Gamma$ and a walk $Y$ which satisfy Assumptions \ref{ass1} - \ref{ass4}.
An upper bound on the Green function $G_N$ is already enough to prove the upper bound of Theorem \ref{th: frequent points, general case}:

\begin{proof}[Proof of the upper bound of Theorem \ref{th: frequent points, general case}]
Let $a \geq 0$ and $N \geq 1$. For every $\eps>0$ we obtain by Markov inequality:
\begin{align*}
\PROB{x_0}{ \abs{\Mc_N(a) } \geq N^{2(1-a) + \eps} }
& \leq N^{-2(1-a) - \eps } \sum_{x \in V_N} \PROB{x_0}{\ell_x^{\tau_N} \geq 2ga (\log N)^2 }.
\end{align*}
But for every $x \in V_N$, under $\prob_x$, $\ell_x^{\tau_N}$ is an exponential variable with mean $G_N(x,x)$. Hence by \eqref{eq: hypothesis upper bound G_N},
\begin{align}
\PROB{x_0}{\ell_x^{\tau_N} \geq 2ga (\log N)^2 } &
= \PROB{x_0}{\ell_x^{\tau_N} > 0} \PROB{x}{\ell_x^{\tau_N} \geq 2ga (\log N)^2 } \nonumber \\
& = \PROB{x_0}{\ell_x^{\tau_N} > 0} \exp \left( - 2ga (\log N)^2 / G_N(x,x) \right) \nonumber \\
& \leq C N^{-2a+o(1)}. \label{eq: proof upper bound}
\end{align}
The upper bound for the convergence in probability follows. To show that
\[ \limsup_{N \rightarrow \infty} \frac{\log \abs{\Mc_N(a)}}{\log N} \leq 2(1-a), \quad \quad \prob_{x_0} \mathrm{-a.s.},
\]
we observe that, taking $N = 2^n$ in \eqref{eq: proof upper bound},
\[
\PROB{x_0}{ \# \left\{ x \in V_{2^{n+1}}: \ell_x^{\tau_{2^{n+1}}} \geq 2ga \left( \log 2^n \right)^2 \right\} \geq (2^n)^{2(1-a)+\eps} }
\]
decays exponentially and so is summable.
Moreover, if $2^n \leq N < 2^{n+1}$,
\[
\abs{\Mc_N(a)} \leq \# \left\{ x \in V_{2^{n+1}}: \ell_x^{\tau_{2^{n+1}}} \geq 2ga \left( \log 2^n \right)^2 \right\}.
\]
Hence the Borel--Cantelli lemma implies that
\[
\limsup_{N \rightarrow \infty} \frac{\log \abs{\Mc_N(a)}}{\log N} \leq 2(1-a)+\eps, \quad \quad \prob_{x_0} \mathrm{-a.s.}
\]
This concludes the proof of the upper bound on $\abs{\Mc_N(a)}$. We notice that the above reasoning also shows that for all $\eps >0$, almost surely, for all $N$ large enough, $\abs{\Mc_N(1+\eps)} = 0$.
The upper bound on $\sup_{x \in V_N} \ell_x^{\tau_N}$ then follows from
\[
\left\{ \sup_{x \in V_N} \ell_x^{\tau_N} \geq 2g(1+\eps) (\log N)^2 \right\}
\subset \left\{ \abs{\Mc_N(1+\eps)} \geq 1 \right\}.
\]
\end{proof}

\subsection{Lower bound}\label{sec: Lower Bound}

We first start this section by establishing a lemma which simplifies a bit the problem: we only need to show that the probability to have a lot of thick points decays sub-polynomially. For all starting point $x_0' \in V_N$, define $\Mc_N(a,x_0')$ the set of $a$-thick points in the ball $V_N(x'_0)$:
\[
\Mc_N(a,x_0') = \left\{ x \in V_N(x_0'): \ell_x^{\tau_N(x_0')} \geq 2 ga (\log N)^2 \right\}.
\]

\begin{lemma}\label{lem: from not too small proba to proba 1}
Suppose that for all starting point $x_0' \in V_N(x_0)$, for all $a \in (0,1), \eps>0$ and $N \in \N$,
\[
\PROB{x_0'}{ \abs{\Mc_N(a,x_0')} \geq N^{2(1-a) - \eps} } \geq p_N,
\]
with $p_N = p_N(a)>0$ decaying slower than any polynomial, i.e. $\log p_N = o_{a,\eps}(\log N)$. Then for all $a \in (0,1)$,
\[
\liminf_{N \rightarrow \infty} \frac{\log \abs{\Mc_N(a)}}{\log N} \geq 2(1-a), \quad \quad \prob_{x_0} \mathrm{-a.s.}
\]
\end{lemma}

\begin{proof}

A similar but weaker statement appears in \cite{dembo2001} and \cite{rosen2006} where they assumed that $p_N$ was bounded away from $0$. The idea is to decompose the walk in the ball $V_N(x_0)$ into several walks in smaller balls to bootstrap the probability we are interested in.

First of all, let us remark that if $p_N \in (0,1)$ decays slower than any polynomial, then so does $\left( \inf_{n \leq N} p_n \right)_{N \geq 1}$. Consequently, we can assume without loss of generality that the sequences $p_N$ in the statement of the lemma are non increasing.

Fix $\eps>0$ and take $N$ large and $K_N \in \N$ much smaller than $N$ such that $K_N = N^{1 - o(1)}$. Let us introduce the stopping times
\[
\sigma(0) := 0 \mathrm{~and~} \forall i \geq 1, \sigma(i) := \inf \left\{ t > \sigma(i-1): d_\Gamma \left( Y_t, Y_{\sigma(i-1)} \right) \geq K_N \right\}
\]
and
\[
i_{\mathrm{max}} := \max \left\{ i \geq 0, d_\Gamma \left( x_0, Y_{\sigma(i)} \right) \leq N - K_N \right\}.
\]
Let $k \geq 1$. If $i_\mathrm{max} + 1 \geq k$, then all the walks $\left( Y_{\sigma(i) + t}, 0 \leq t \leq \sigma(i+1) - \sigma(i) \right)$, $i= 0 \dots k-1$, are contained in the walk $\left( Y_t, 0 \leq t \leq \tau_N \right)$. So by a repeated application of Markov property, we see that for all $\delta>0$, if $N$ is large enough so that $a (\log N)^2 \leq (a + \delta) (\log K_N)^2$ (which is possible by assumption on $K_N$), we have:
\begin{align*}
& \prob_{x_0} \Big( \abs{\Mc_N(a)} \leq N^{2(1 -a) - \eps} \Big) \\
& \leq \sup_{x_0' \in V_{N-K_N}(x_0)} \PROB{x_0'}{ \abs{\Mc_{K_N}(a+\delta,x_0')} \leq N^{2(1-a) - \eps} }^k + \PROB{x_0}{i_\mathrm{max} + 1 \leq k} \\
& \leq \sup_{x_0' \in V_{N-K_N}(x_0)} \PROB{x_0'}{ \abs{\Mc_{K_N}(a+\delta,x_0')} \leq K_N^{(2(1-a) - \eps)\sqrt{1 + \delta/a}} }^k + \PROB{x_0}{i_\mathrm{max} + 1 \leq k}.
\end{align*}
If $\delta >0$ is small enough we have $ (2(1-a) - \eps) \sqrt{1 + \delta / a} < 2(1-a-\delta) $. Hence with $p_N = p_N(a + \delta)$
\begin{align}
\PROB{x_0}{ \abs{\Mc_N(a)} \leq N^{ 2(1-a) - \eps} } & \leq (1 - p_{K_N})^k + \PROB{x_0}{i_\mathrm{max} + 1 \leq k} \nonumber \\
& \leq (1 - p_N)^k + \PROB{x_0}{i_\mathrm{max} + 1 \leq k}.\label{eq: proof bootstrap}
\end{align}
To conclude, we have to choose $K_N$ small enough to ensure that $i_\mathrm{max}$ is large with high probability. If the walk were a nearest neighbour random walk, we could say that $i_\mathrm{max} + 1 \geq \floor{N / K_N} ~\prob_{x_0}$-a.s. Here, the jumps may be unbounded but large jumps are costly (Assumption \ref{ass4}) so we will be able to recover a lower bound fairly similar on $i_\mathrm{max}$. By the triangle inequality, we have for all $k \geq 1$
\begin{align*}
\PROB{x_0}{i_\mathrm{max} + 1 \leq k}
& \leq 
\PROB{x_0}{ \exists i \leq k-1, d_\Gamma \left( Y_{\sigma(i)}, Y_{\sigma(i+1)} \right) \geq (N-K_N)/k } \\
& \leq
\sum_{i = 0}^{k-1} \PROB{x_0}{ Y_{\sigma(i)} \in V_{N - K_N}, d_\Gamma \left( Y_{\sigma(i)}, Y_{\sigma(i+1)} \right) \geq (N-K_N)/k } \\
& \leq
k \sup_{x_0' \in V_{N-K_N}} \PROB{x'_0}{ d_\Gamma \left( x_0', Y_{\tau_{K_N}} \right) \geq (N-K_N)/k }.
\end{align*}
Assumption \ref{ass4} allows us to bound this last probability: there exists $(\eps_N)_{N \geq 1} \subset (0, \infty)$ which converges to zero such that if $M >0$,
\begin{align*}
\PROB{x'_0}{ d_\Gamma \left( x_0', Y_{\tau_{K_N}} \right) \geq M + K_N }
& \leq
K_N N^{\eps_N} / M.
\end{align*}
Hence
\[
\PROB{x_0}{i_\mathrm{max} + 1 \leq k} \leq \frac{k^2 K_N N^{\eps_N}}{N-(k+1)K_N}.
\]
Coming back to the estimate \eqref{eq: proof bootstrap} and taking $k = (\log N)/p_N$ , we have obtained
\begin{align*}
\PROB{x_0}{ \abs{\Mc_N(a)} \leq N^{ 2(1-a) - \eps} }
& \leq
(1-p_N)^{(\log N) / p_N} + \PROB{x_0}{i_\mathrm{max} + 1 \leq (\log N) / p_N} \\
\leq &
\left( \sup_{0 < p < 1} (1-p)^{1/p} \right)^{\log N} + C \frac{(\log N)^2 K_N N^{\eps_N}}{(p_N)^2(N-(1+(\log N)/p_N)K_N)}.
\end{align*}
We can choose
\[
K_N = \frac{p_N^2}{(\log N)^4} N^{1-\eps_N}  = N^{1-o(1)}
\]
so that the previous estimates gives
\[
\PROB{x_0}{ \abs{\Mc_N(a)} \leq N^{ 2(1-a) - \eps} } \leq C / (\log N)^2.
\]
We now conclude as in the proof of the upper bound of Theorem \ref{th: frequent points, general case}. We apply the Borel--Cantelli lemma along the sequence $(2^p)_{p \in \N}$ which yields
\[
\liminf_{p \rightarrow \infty} \frac{\log \abs{\Mc_{2^p}(a)}}{\log \left( 2^p \right)} \geq 2(1-a), \quad \quad \prob_{x_0} \mathrm{-a.s.}
\]
This finishes the proof of the lemma because $\log \left(2^{p+1} \right) / \log \left( 2^p \right) \to 1$ as $p \to \infty$.
\end{proof}

As mentioned at the end of Section \ref{sec: Literature Overview and Organization of the Paper}, when we will use Eisenbaum's isomorphism, we will have to bound from above expectations of the form:
\[
\mathds{E} \left[ 1 + \frac{\phi_N(x_0)}{s} ; A \right] := \mathds{E} \left[ \left( 1 + \frac{\phi_N(x_0)}{s} \right) \mathbf{1}_A \right]
\]
for some given event $A$. We will use the following elementary lemma which we state here only for convenience:

\begin{lemma}\label{lem: upper bound event Eisenbaum}
For all $N$ large enough and for all events $A$,
\[
\mathds{E} \left[ \left( 1 + \frac{\phi_N(x_0)}{s} \right) ; A \right] \leq (\log N)^2 \mathds{P}(A) + N^{-\log N}.
\]
\end{lemma}

\begin{proof}
Using \eqref{eq: hypothesis upper bound G_N}, we have:
\begin{align*}
\mathds{E} \left[ \left( 1 + \frac{\phi_N(x_0)}{s} \right) ; A \right]
& \leq \left( \log N \right)^2 \mathds{P}(A) +
\mathds{E} \left[ \left( 1 + \frac{\phi_N(x_0)}{s} \right) \indic{ 1+\phi_N(x_0)/s \geq (\log N)^2} \right] \\
& \leq \left( \log N \right)^2 \mathds{P}(A) + \exp \left( -\frac{s^2}{2g} (\log N)^3(1+o(1)) \right),
\end{align*}
which concludes the lemma.
\end{proof}

We now provide our proof of the lower bound of Theorem \ref{th: frequent points, general case}. In the following, we write our arguments with the starting point $x_0$ but note that the same also works for all starting points $x_0' \in V_N(x_0)$, which is required to apply Lemma \ref{lem: from not too small proba to proba 1}.

\begin{proof}[Proof of the lower bound of Theorem \ref{th: frequent points, general case}]
During the entire proof we will fix some small $\eta>0$. To ease notations, we will denote $Q_N := Q_N(x_0)$.
Recall that if $x \in Q_N$ and $1 \leq R \leq N^{1-\eta}$, Assumption \ref{ass3} gives the existence of a subset $C_R(x) \subset Q_N$ which can be thought of as a circle of radius $R$ around $x$. We will denote $M_R^x$ the operator corresponding to taking the mean value of a function on this circle: if $f$ is a function defined on $Q_N$, then
\[
M_R^x f = \frac{1}{\# C_R(x)} \sum_{y \in C_R(x)} f(y) \in \R.
\]

We use Eisenbaum's isomorphism with some $s>0$ ($s = 1$ will do). Let $\eps_N = 1/ \sqrt{\log N}$ and for some $b>a$ (to be chosen later on, close to $a$) and $\phi_N$ a GFF independent of the walk, we define the good events at $x$:
\begin{align*}
G_N^{b,\eta}(x,\ell^{\tau_N})
& = \left\{ M_R^x \ell^{\tau_N } \leq 2g b \left( \log \frac{N}{R} \right)^2, \forall R \in (2^p)_{p \in \N} \cap \left\{ 1, \dots, N^{1-\eta} \right\} \right\}, \\
G_N^{\eta}(x,\phi_N)
& = \left\{ M_R^x \left(\frac{1}{2} (\phi_N+s)^2 \right) \leq \eps_N \left( \log \frac{N}{R} \right)^2, \forall R \in (2^p)_{p \in \N} \cap \left\{ 1, \dots, N^{1-\eta} \right\} \right\},
\end{align*}
and
\begin{equation}\label{eq:Good events, general case}
G_N^{b,\eta}(x) = G_N^{b,\eta}(x,\ell^{\tau_N}) \cap G_N^{\eta}(x,\phi_N).
\end{equation}
We require the points to be never to thick at any scales (similar to \cite{berestycki2017}).
We restrict ourselves to $Q_N$ (the subset of $V_N$ where we control the Green function $G_N$) by considering:
\[
\widetilde{\Mc}_N(a) = \Mc_N(a) \cap Q_N
\]
and we will abusively write $\abs{\widetilde{\Mc}_N(a) \cap G_N^{b,\eta}} $ when we mean $ \sum_{x \in Q_N} \indic{x \in \widetilde{\Mc}_N(a)} \indic{G_N^{b,\eta}(x)}$. The Paley--Zigmund inequality gives:
\begin{align*}
\PROB{x_0}{ \abs{ \Mc_N(a) } \geq \frac{1}{2}\EXPECTprod{x_0}{ \abs{\widetilde{\Mc}_N(a) \cap G_N^{b,\eta}} } }
& \geq \frac{1}{4} \frac{\EXPECTprod{x_0}{ \abs{\widetilde{\Mc}_N(a) \cap G_N^{b,\eta}} }^2}{\EXPECTprod{x_0}{ \abs{\widetilde{\Mc}_N(a) \cap G_N^{b,\eta}}^2 }}
\end{align*}
and it remains to estimate the first and second moments on the right hand side.

\paragraph{First Moment Estimate}\label{paragraph: first moment estimate, general case}
Firstly, we estimate the first moment without restricting to any event. Thanks to assumptions \eqref{eq: hypothesis G_N on Q_N} and \eqref{eq: hypothesis G_N starting point} and because, starting from $x$, the law of $\ell_x^{\tau_N}$ is exponential, we have:
\begin{align*}
\EXPECT{x_0}{ \abs{\widetilde{\Mc}_N(a)} } & = \sum_{x \in Q_N} \PROB{x_0}{ \ell_x^{\tau_N} \geq 2ga (\log N)^2 }
= \sum_{x \in Q_N} \frac{G_N(x_0,x)}{G_N(x,x)} \PROB{x}{\ell_x^{\tau_N} \geq 2ga (\log N)^2 } \\
& = \sum_{x \in Q_N} \frac{G_N(x_0,x)}{G_N(x,x)} \exp \left( - \frac{2ga (\log N)^2}{G_N(x,x)} \right) = N^{2 - 2a + o(1)}.
\end{align*}

To estimate the probability $\mathds{P} \left( G_N^{\eta}(x, \phi_N) \right)$ we will first derive a large deviation estimate for $M^x_R \left( (\phi_N+s)^2 \right)$. The estimate we obtain is rough and does not take into account the fact that if $R$ is large we should expect $M^x_R \left( (\phi_N+s)^2 \right)$ to be close to its mean. Writing $\Nc(\mu,\sigma^2)$ a Gaussian variable with mean $\mu$ and variance $\sigma^2$, by Jensen's inequality we have $\forall \lambda >0$ and $\forall t \in (0,1/(2g))$
\begin{align*}
\mathds{P} \left( M_R^x \left( (\phi_N+s)^2 \right) \geq \lambda \log N \right)
& \leq e^{-t \lambda} \mathds{E} \left[ \exp \left( \frac{t}{\log N} M_R^x \left( (\phi_N+s)^2 \right) \right) \right] \\
& \leq e^{-t \lambda} \frac{1}{\# C_R(x)} \sum_{y \in C_R(x)} \mathds{E} \left[ \exp \left( \frac{t}{\log N} (\phi_N(y)+s)^2 \right) \right] \\
& \leq e^{-t \lambda} \mathds{E} \left[ \exp \left\{ (tg+o(1)) \Nc(o(1),1+o(1))^2 \right\} \right]
\leq C(t) e^{-t \lambda}
\end{align*}
where $0 < C(t) < \infty$ because $tg$ is smaller than $1/2$. Hence, we have obtained: for all $t \in \left(0, 1/(2g) \right)$, there exists $C(t) \in (0, \infty)$ such that
\begin{equation}\label{eq: Tail bound on M_R^x phi^2}
\forall x \in Q_N, \forall 1 \leq R \leq N^{1-\eta}, \forall \lambda > 0,
\mathds{P} \left( M_R^x \left( (\phi_N+s)^2 \right) \geq \lambda \log N \right) \leq C(t) e^{- t \lambda}.
\end{equation}
Hence, using the above estimate with $t = 1/(4g)$ for instance, if $x \in Q_N$, the probability that the good event at $x$ linked to $\phi_N$ does not hold is:
\begin{align*}
\mathds{P} \left( G_N^{\eta}(x, \phi_N)^c \right)
& \leq \sum_{\substack{R = 2^p, ~p \in \N \\ 1 \leq R \leq N^{1-\eta}}} \mathds{P} \left( M_R^x \left( \frac{1}{2} (\phi_N+s)^2 \right) > \eps_N \left( \log \frac{N}{R} \right)^2 \right) \\
& \leq \sum_{\substack{R = 2^p, ~p \in \N \\ 1 \leq R \leq N^{1-\eta}}} \mathds{P} \left( M_R^x \left( \frac{1}{2} (\phi_N+s)^2 \right) > \eta^2 \eps_N (\log N)^2 \right) \\
& \leq \exp \left( - C(\eta) \eps_N \log N \right) \xrightarrow[ N \rightarrow \infty]{} 0
\end{align*}
for some $C(\eta) > 0$. By independence of $\phi_N$ and the local times of the random walk, we thus have
\[
\PROBprod{x_0}{\ell_x^{\tau_N} \geq 2ga (\log N)^2, G_N^{b,\eta}(x)}
= \left(1-o_\eta(1) \right) \PROB{x_0}{\ell_x^{\tau_N} \geq 2ga (\log N)^2, G_N^{b,\eta}(x,\ell^{\tau_N})}.
\]
Now, using the Eisenbaum's isomorphism and Lemma \ref{lem: upper bound event Eisenbaum}, we can bound from above the probability $\PROB{x_0}{\ell_x^{\tau_N} \geq 2g a ( \log N)^2, G_N^{b,\eta} \left( x, \ell^{\tau_N} \right)^c } $, for a given $x \in Q_N$, by the sum over $R \in \{2^p, p \in \N\} \cap [1,N^{1-\eta}]$ of
\begin{align*}
& \PROB{x_0}{ \ell_x^{\tau_N} \geq 2g a ( \log N)^2, M_R^x \left( \ell^{\tau_N} \right) \geq 2g b \left( \log \frac{N}{R} \right)^2 } \\
& \leq \mathds{E} \bigg[ \left( 1+ \frac{\phi_N(x_0)}{s} \right) ; \abs{ \phi_N(x) + s }^2 \geq 4g a (\log N)^2, M_R^x \left(\abs{ \phi_N + s }^2 \right) \geq 4g b \left( \log \frac{N}{R} \right)^2 \bigg] \\
& \leq (\log N)^2 \mathds{P} \left( \abs{ \phi_N(x) + s }^2 \geq 4g a (\log N)^2, M_R^x \left( \abs{ \phi_N + s }^2 \right) \geq 4g b \left( \log \frac{N}{R} \right)^2 \right) \\
& ~~~~+ O \left( N^{-\log N} \right).
\end{align*}
By taking $\delta = 2 \sqrt{a/g}$, we can bound from above the probability 
appearing in the last equation by:
\begin{align*}
& (2+o(1)) \mathds{P} \left( \phi_N(x) \geq (2\sqrt{ga} + o(1)) \log N, M_R^x \left( \abs{ \phi_N + s }^2 \right) \geq 4g b \left( \log \frac{N}{R} \right)^2 \right) \\
& = (2+o(1)) \mathds{P} \left( e^{\delta \phi_N(x)} \indic{ M_R^x ((\phi_N+s)^2) \geq 4gb \left( \log \frac{N}{R}  \right)^2} \geq N^{ 2 \sqrt{ga} \delta + o(1) } \right) \\
& \leq N^{-4a+o(1)} \mathds{E} \left[ e^{\delta \phi_N(x)} \indic{ M_R^x ((\phi_N+s)^2) \geq 4gb \left( \log \frac{N}{R}  \right)^2} \right] \\
& = N^{-4a + o(1)} e^{ \frac{\delta^2}{2} \mathds{E} \left[\phi_N(x)^2 \right] } \widetilde{\mathds{P}} \left( M_R^x ((\phi_N+s)^2) \geq 4gb \left( \log \frac{N}{R} \right)^2 \right)
\end{align*}
where $\widetilde{\mathds{P}}$ is the shifted probability:
\[
\frac{d \widetilde{\mathds{P}}}{d \mathds{P}} = e^{\delta \phi_N(x) - \frac{\delta^2}{2}\mathds{E} \left[ \phi_N(x)^2 \right] }.
\]
By Cameron--Martin theorem, under this new probability, $\phi_N$ has the same covariance structure but the mean of $\phi_N(y)$ is now given by:
\[
\mathrm{Cov}_{\mathds{P}}(\phi_N(y), \delta \phi_N(x)) = \left( 2 \sqrt{ga} + o_\eta(1) \right) \log \frac{N}{d_\Gamma(x,y)} = ( 2 \sqrt{ga} + o_\eta(1)) \log \frac{N}{R} \mathrm{~if~} y \in C_R(x).
\]
As we have taken $b>a$, we can apply our tail estimate \eqref{eq: Tail bound on M_R^x phi^2} to show that,
\begin{align*}
\PROB{x_0}{\ell_x^{\tau_N} \geq 2g a ( \log N)^2, G_N^{b,\eta} \left( x, \ell^{\tau_N} \right)^c } & \leq 
N^{-2a - t + o(1)}
\end{align*}
for some small $t>0$ which may depend on $\eta, a$ and $b$.
With the estimate on the first moment without the event $G_N^{b,\eta}$, this shows that:
\[
\EXPECTprod{x_0}{ \abs{\widetilde{\Mc}_N(a) \cap G_N^{b,\eta}} } \geq N^{2(1-a)+o(1)}.
\]

\paragraph{Second Moment Estimate}
To control the second moment, we adapt the ideas of \cite{berestycki2017} to our framework: let $x,y \in Q_N$ such that $d_\Gamma(x,y) \leq N^{1-\eta}$. We can find some $R \in (2^p)_{p \in \N}, R \leq N^{1-\eta}$ such that
\[
\frac{1}{2} \left( d_\Gamma(x,y) \vee 1 \right) \leq R \leq d_\Gamma(x,y) \vee 1.
\]
As before, we apply the Eisenbaum isomorphism, Lemma \ref{lem: upper bound event Eisenbaum}, an exponential Markov inequality, and using the fact that by Cauchy--Schwarz $\abs{M_R^x \phi_N} \leq \sqrt{M_R^x ((\phi_N+s)^2)}+s$, we have:
\begin{align}
& \PROBprod{x_0}{ \ell_x^{\tau_N} \mathrm{~and~} \ell_y^{\tau_N} \geq 2g a (\log N)^2, G_N^{b,\eta}(x), G_N^{b,\eta}(y) } \nonumber \\
& \leq (2+o(1)) (\log N)^2 \mathds{P} \Bigg( \phi_N(x) \mathrm{~and~} \phi_N(y) \geq \left( 2 \sqrt{ga} + o(1) \right) \log N, \nonumber \\
& ~~~~~~~~~~~~~~~~~~~~~~~~~~~~~~~~ M_R^x \phi_N \leq \left( 2 \sqrt{gb} + o_\eta(1) \right) \log \frac{N}{R} \Bigg) + N^{-\log N} \nonumber \\
& \leq N^{-4a + o(1)} \left( \frac{N}{d_\Gamma(x-y) \vee 1} \right)^{4a} \widetilde{\mathds{P}} \left( M_R^x \phi_N \leq \left( 2 \sqrt{gb} + o_\eta(1) \right) \log \frac{N}{R} \right) + N^{-\log N} \label{eq: proba second moment}
\end{align}
where $\widetilde{\mathds{P}}$ denotes the shifted probability defined by
\[
\frac{d \widetilde{\mathds{P}}}{d \mathds{P}} = e^{\delta \phi_N(x) + \delta \phi_N(y) - \frac{\delta^2}{2} \mathds{E} \left[(\phi_N(x) + \phi_N(y))^2 \right]} \mathrm{~with~} \delta = 2 \sqrt{\frac{a}{g}}.
\]
By Cameron--Martin theorem, under the probability $\widetilde{\mathds{P}}, \phi_N$ has the same covariance structure but the mean of $\phi_N(z)$ is now given by:
\[
\mathrm{Cov}_{\mathds{P}}(\phi_N(z), \delta \phi_N(x) + \delta \phi_N(y)) = \left( 4 \sqrt{ga} + o_\eta(1) \right) \log \frac{N}{R} \mathrm{~if~} z \in C_R(x)
\]
by our particular choice of R. Thanks to Assumptions \eqref{eq: hypothesis G_N on Q_N} and \eqref{eq: hypothesis existence circles, control sum}, one can check that the variance of $M_R^x \phi_N$ is equal to $\left( g + o_\eta(1) \right) \log \frac{N}{R}$. Hence
\begin{align*}
& \widetilde{\mathds{P}} \left( M_R^x \phi_N \leq \left( 2 \sqrt{gb} + o_\eta(1) \right) \log \frac{N}{R} \right) \\
& \leq 
\mathds{P} \left( \mathcal{N}(0,1) \leq - \left( 2 (2 \sqrt{a} - \sqrt{b}) + o_\eta(1) \right) \sqrt{\log \frac{N}{R}} \right) \\
& \leq \left( \frac{N}{R} \right) ^{-2 (2 \sqrt{a} - \sqrt{b})^2 + o_\eta(1) }.
\end{align*}
Again thanks to our particular choice of $R$, we have obtained:
\begin{align*}
& \PROBprod{x_0}{ \ell_x^{\tau_N}, \ell_y^{\tau_N} \geq 2g a (\log N)^2, G_N^{b,\eta}(x), G_N^{b,\eta}(y) } \\
& \leq N^{-4a+o_\eta(1)} \left( \frac{N}{d_\Gamma(x,y) \vee 1} \right)^{4a - 2(2\sqrt{a}-\sqrt{b})^2}.
\end{align*}
As $a < 1$, we can choose $b>a$ close enough to $a$ to ensure that the exponent $4a - 2(2 \sqrt{a}-\sqrt{b})^2$ is less than $2$. We can then sum over all $x,y \in Q_N$ such that $\abs{x-y} \leq N^{1-\eta}$ and use assumption \eqref{eq: hypothesis on Q_N 2} to find that:
\begin{align*}
\EXPECTprod{x_0}{ \abs{\widetilde{\Mc}_N(a) \cap G_N^{b,\eta}}^2 } & \leq N^{4(1-a) + o_\eta(1) } + \sum_{ \substack{x,y \in Q_N \\ d_\Gamma(x,y) \geq N^{1-\eta}}} \PROB{x_0}{\ell_x^{\tau_N}, \ell_y^{\tau_N}  \geq 2ga (\log N)^2}.
\end{align*}
We eventually treat our last sum noticing that the probability in this sum is not larger than (using \eqref{eq: proba second moment} without the term $\tilde{\prob}(\cdots)$):
\[
N^{-4a + o(1)} \left( \frac{N}{d_\Gamma(x,y)} \right)^{4a} \leq N^{-4a + 4a \eta + o(1) }.
\]
This shows that the second moment is not larger than $N^{4(1-a+a\eta)+o_\eta(1)}$. To come back to the probability we wanted to bound from below, this implies:
\[
\PROB{x_0}{\abs{\Mc_N(a)} \geq N^{2(1-a)+o(1)} } \geq N^{-4a \eta +o_\eta(1)}.
\]
As this is true for all $\eta>0$, it means that the probability is not less than $(1/N)^{o(1)}$. We can then use Lemma \ref{lem: from not too small proba to proba 1} to conclude the proof of Theorem \ref{th: frequent points, general case}.
\end{proof}

\section{Higher dimensions}\label{sec:Dimension 3}

\subsection{Proofs of Theorems \ref{th: dim 3 convergence measures}, \ref{th: dim 3 convergence thick points} and \ref{th: dim 3 convergence supremum} }\label{sec: dim 3 proofs theorems}

This section is devoted to the proofs of Theorems \ref{th: dim 3 convergence measures}, \ref{th: dim 3 convergence thick points} and \ref{th: dim 3 convergence supremum}. Let us first recall the setting and introduce some new notations.
Consider a continuous time (rate 1) random walk $(Y_t)_{t \geq 0}$ on $\Z^d$ for $d \geq 3$ and denote $\prob_x$ and $\expect_x$ its law and expectation starting from $x$. Writing $V_N = \{ -N, \dots, N \}^d$, we consider the first exit time of $V_N$ and the first hitting time of $x$:
\begin{equation}\label{eq: dim 3 tau_N tau_x}
\tau_N := \inf \{ t \geq 0, Y_t \notin V_N \}, \forall x \in \Z^d, \tau_x := \inf \{ t \geq 0: Y_t = x \}.
\end{equation}
We will denote $G$ and $G_N$ the Green function on $\Z^d$ and on $V_N$ respectively: for all $x,y \in \Z^d$,
\begin{equation}\label{eq: dim 3 def Green functions}
G(x,y) := \EXPECT{x}{\int_0^\infty \indic{Y_t = y} dt } \mathrm{~and~} G_N(x,y) := \EXPECT{x}{\int_0^{\tau_N} \indic{Y_t = y} dt }.
\end{equation}
Finally, we denote $g := G(0,0)$ the value of $G$ on the diagonal and $\omega(x, dz)$ the harmonic measure on $[-1,1]^d$: for all $x \in [-1,1]^d, E \subset \partial [-1,1]^d, \omega(x,E)$ denotes the probability that a Brownian motion starting from $x$ exits $[-1,1]^d$ through $E$. In the following, if $x \in \R^d$, we will denote $\floor{x}$ one element of $\Z^d$ which is closest to $x$.

Let us first recall the behaviour of $G_N$ in dimension greater or equal to $3$:

\begin{lemma}\label{lem: behaviour Green function dimension >= 3}
For all $\eta \in (0,1)$, we have the following estimates:
\begin{align*}
\forall x \in V_N, G_N(x,x) & \leq g, \\
\forall x \in V_{(1-\eta)N}, G_N(x,x) & \geq g + O_\eta \left( N^{2-d} \right).
\end{align*}
Moreover, if $a_d = d/2 ~ \Gamma(d/2 - 1) \pi^{-d/2}$, we have for all $x \neq y \in V_N$,
\[
G_N(x,y) = a_d \left( \abs{x-y}^{2-d} - q_N(x,y) \right)
\]
where $q_N(x,y) \geq O \left( \abs{x-y}^{-d} \right)$ and for all $\tilde{x}, \tilde{y} \in (-1,1)^d$, we have the following pointwise estimate:
\begin{equation}\label{eq:lem pointwise limit Green}
\lim_{N \rightarrow \infty} N^{d-2} q_N \left( \floor{ N \tilde{x} }, \floor{ N \tilde{y} } \right)
= 
\int_{\partial [-1,1]^d} \abs{ \tilde{y} - \tilde{z} }^{2-d} \omega(\tilde{x},d \tilde{z}) =: q( \tilde{x}, \tilde{y}).
\end{equation}
\end{lemma}

The proof of this lemma will be given in Section \ref{sec: dim 3 proof lemmas}. As mentioned in Section \ref{sec: Literature Overview and Organization of the Paper}, a key point is to show that all the moments of the number of thick points converge which is the purpose of the next proposition. Before stating it, let us introduce some notations.

\textbf{Notation:}
If $k \geq 1$ and $q \geq 1$, we denote by $f(k \to q)$ the number of ways to partition a set with $k$ elements into $q$ non empty sets. As this is equal to the number of surjective functions from $\{1 \dots k\}$ to $\{1 \dots q\}$ divided by $q!$, we have
\begin{equation}\label{eq:dim3 def f}
f(k \to q) = \frac{1}{q!} \sum_{i=1}^q \binom{q}{i}(-1)^{q-i} i^k.
\end{equation}
If $X$ is a topological space we will denote by $\Bc(X)$ the class of Borel sets of $X$.

\begin{proposition}\label{prop: dim 3 moment number thick points}
Let $r \geq 1$ and for all $i=1 \dots r$, take $k_i \geq 1, A_i \in \Bc([-1,1]^d)$ such that the Lebesgue measure of $\bar{A_i} \backslash A_i^\circ$ vanishes, $T_i \in \Bc(\R)$ with $\inf T_i > - \infty$. Moreover, we assume that the $A_i \times T_i$'s are pairwise disjoint. By denoting $k = k_1 + \cdots + k_r$ we define
\begin{align}\label{eq:def m(AxT)}
m ( A_i \times T_i, k_i&, i = 1 \dots r ) :=
\left( \frac{a_d}{g} \right)^k \prod_{i=1}^r \left( \int_{T_i} e^{-t/g} \frac{dt}{g} \right)^{k_i} \\
& \times \sum_{\sigma \in \mathfrak{S}_k} \int_{A_1^{k_1} \times \dots \times A_r^{k_r}} \prod_{i=0}^{k-1} \left( \abs{y_{\sigma(i+1)} - y_{\sigma(i)}}^{2-d} - q \left( y_{\sigma(i)}, y_{\sigma(i+1)} \right) \right) dy_1 \dots dy_k \nonumber
\end{align}
with the convention $y_{\sigma(0)} = 0$. 

1. Subcritical regime: let $a \in [0,1)$ and if $a=0$ assume furthermore that $T_i \subset (0,\infty)$ for all $i$. Then
\begin{equation}\label{eq:prop subcritical regime}
\lim_{N \rightarrow \infty} \EXPECT{0}{ \prod_{i=1}^r \left\{ \nu_N^a \left( A_i \times T_i \right) \right\}^{k_i} }
=
m(A_i \times T_i, k_i, i = 1 \dots r).
\end{equation}

2. At criticality, 
\begin{align}
\lim_{N \rightarrow \infty} & \EXPECT{0}{ \prod_{i=1}^r \left\{ \nu_N^1 \left( A_i \times T_i \right) \right\}^{k_i} } \nonumber \\
& =
\sum_{\substack{1 \leq q_i \leq k_i \\ i = 1 \dots r }} \left( \prod_{i=1}^r f( k_i \to q_i) \right) m \left( A_i \times T_i, q_i, i = 1 \dots r \right). \label{eq:prop criticality}
\end{align}

The previous results also hold if we replace $\nu_N^a$ by $\mu_N^a$.
\end{proposition}

We postpone the proof of this proposition to the next section and we now explain how we can deduce Theorems \ref{th: dim 3 convergence measures}, \ref{th: dim 3 convergence thick points} and \ref{th: dim 3 convergence supremum} from it. We start with Theorem \ref{th: dim 3 convergence measures}.

\begin{proof}[Proof of Theorem \ref{th: dim 3 convergence measures}]
This proof will be decomposed in three small parts. First, we will show that the previous proposition implies the joint convergence of $(\nu_N^a(A_1 \times T_1), \dots, \nu_N^a(A_r \times T_r))$ with suitable $A_i$'s and $T_i$'s. The second part is relatively standard and shows that it then implies the convergence in law of the sequence of random measures $\{ \nu_N^a, N \geq 1\}$. The third part is dedicated to the identification of the limiting measures.

\emph{Step 1.} Take $a \in [0,1]$. Let us first show that the previous proposition implies the convergence of the joint distribution $(\nu_N^a(A_1 \times T_1), \dots, \nu_N^a(A_r \times T_r))$ where the $A_i$'s and $T_i$'s are as in the statement of the proposition. As all their moments converge, we just need to check that the limiting moments do not grow too rapidly. Take $k_1 \dots k_r \geq 1$. We notice that for all $x \in [-1,1]^d$,
\begin{align*}
0 \leq \int_{[-1,1]^d} \left( \abs{y - x}^{2-d} - q(x,y) \right) dy
& \leq \int_{[-1,1]^d} \abs{y - x}^{2-d} dy \\
& \leq \int_{[-2+x,2+x]^d} \abs{y - x}^{2-d} dy = C
\end{align*}
for some universal constant $C$ depending only on the dimension $d$. Hence there exists $C'$ depending on $d$ and on the $T_i$'s such that
\begin{equation}\label{eq: dim 3 proof thm convergence measures}
m(A_i \times T_i, k_i, i = 1 \dots r) \leq C'^k k!
\end{equation}
with $k = k_1 + \dots + k_r$. In particular, it implies that the moment generating function associated to those moments has a positive radius of convergence and they determine a unique law. It thus proves the claimed convergence in the subcritical regime. At criticality, we notice that for all $q \leq k$,
\[
\sum_{\substack{1 \leq q_i \leq k_i \\ i = 1 \dots r }} \indic{q_1 + \dots + q_r = q} \prod_{i=1}^r f( k_i \to q_i)
\]
is not larger than the number of ways to partition a set of $k$ elements into no more than $q$ parts which is equal to $q^k / (q!)$.
Using \eqref{eq: dim 3 proof thm convergence measures}, it implies that
\begin{align*}
\sum_{\substack{1 \leq q_i \leq k_i \\ i = 1 \dots r }} & \left( \prod_{i=1}^r f( k_i \to q_i) \right) m \left( A_i \times T_i, q_i, i = 1 \dots r \right)
\\
& \leq 
\sum_{q=r}^k C'^q q! \sum_{\substack{1 \leq q_i \leq k_i \\ i = 1 \dots r }} \indic{q_1 + \dots + q_r = q} \prod_{i=1}^r f( k_i \to q_i) \leq \sum_{q=r}^k C'^q q^k
\leq C'^k k^{k+1} \leq \tilde{C}^k k!.
\end{align*}
Again the radius of convergence of the associated moment generating function is positive and it gives the required convergence in the critical case as well. We will denote $\nu^a(A_1 \times T_1), \dots, \nu^a(A_r \times T_r)$ random variables which have the limiting distribution of $(\nu_N^a(A_1 \times T_1), \dots, \nu_N^a(A_r \times T_r))$.

\emph{Step 2.} We now show the convergence of the sequence of random measures $\{ \nu_N^a, N \geq 1 \}$. Recalling that the underlying topology is the topology of vague convergence, it is enough to show that for all function $\phi : [-1,1]^d \times \R \rightarrow [0, \infty)$ which are $\Cc^\infty$ with compact support (included in $[-1,1]^d \times (0,\infty)$ if $a=0$),
\[
\scalar{\nu_N^a, \phi} := \int_{[-1,1]^d \times \R} \phi(x,t) d \nu_N^a(x,t)
\]
converges in distribution. It is enough to check that for all $L$-Lipschitz function $h : \R \rightarrow \R$, $\EXPECT{0}{h( \scalar{\nu_N^a, \phi})}$ converges. By Lemma \ref{lem: decomposition function in simple functions}, we can uniformly approximate $\phi$ by a sequence of functions $(\phi_p)_{p \geq 1}$ taking the following form:
\[
\phi_p = \sum_{i=1}^p a_i^{(p)} \mathbf{1}_{ A_i^{(p)} \times T_i^{(p)} }
\]
where $A_i^{(p)} \in \Bc ([-1,1]^d)$ with the Lebesgue measure of $\bar{A}^{(p)}_i \backslash (A_i^{(p)})^\circ$ vanishing, $T_i^{(p)} \in \Bc( \R )$ with $\inf T_i^{(p)} > - \infty$ ($\inf T_i^{(p)} >0$ if $a=0$) and $a_i^{(p)} \in \C$. By the joint convergence proven in Step 1, for all $p \geq 1$,
\[
\lim_{N \rightarrow \infty} \scalar{\nu_N^a, \phi_p} \overset{\mathrm{(d)}}{=} \scalar{\nu^a, \phi_p}
\]
and we can define the law (by dominated convergence theorem for instance)
\[
\scalar{\nu^a, \phi} \overset{\mathrm{(d)}}{:=} \lim_{p \rightarrow \infty} \scalar{ \nu^a, \phi_p}.
\]
We are going to show that we can exchange the two limits, i.e. that $\scalar{\nu^a_N, \phi}$ converges in law to $\scalar{\nu^a, \phi}$.
Recalling that $h$ is $L$-Lipschitz, $\abs{\EXPECT{0}{ h(\scalar{\nu_N^a, \phi}) } - \EXPECT{0}{ h(\scalar{\nu^a, \phi}) } }$ is not larger than
\begin{align*}
&
\abs{ \EXPECT{0}{ h \left( \scalar{\nu_N^a, \phi_p} \right) } - \EXPECT{0}{ h \left( \scalar{\nu^a, \phi_p} \right) } } +
L \EXPECT{0}{ \scalar{\nu_N^a, \abs{\phi - \phi_p} } } \\
& + \abs{ \EXPECT{0}{ h(\scalar{\nu^a, \phi}) } - \EXPECT{0}{ h \left( \scalar{\nu^a, \phi_p} \right) } }.
\end{align*}
By the first part of the proof, the first term goes to zero as $N$ goes to infinity. If $t_0 \in \R$ is such that the support of $\phi$ is included in $[-1,1]^d \times (t_0, \infty)$, then the second term is not larger than
\[
L \norme{\phi-\phi_p}_\infty \EXPECT{0}{\nu_N^a([-1,1]^d \times (t_0,\infty))} \xrightarrow[N \rightarrow \infty]{} L 2^{-p} \EXPECT{0}{\nu^a([-1,1]^d \times (t_0,\infty))}.
\]
Thus the limit of the second term goes to zero when $p \to \infty$. The third term goes to zero by definition and we have proved
\[
\lim_{N \rightarrow \infty} \EXPECT{0}{ h(\scalar{\nu_N^a, \phi}) }  = \EXPECT{0}{ h(\scalar{\nu^a, \phi}) }.
\]

\emph{Step 3.} The convergence of the sequence of random measures $\{ \nu_N^a, N \geq 1\}$ has thus been proved. We are now going to identify the limit. What we did in Step 1 and Step 2 shows that the limiting distribution is entirely determined by the limiting moments from Proposition \ref{prop: dim 3 moment number thick points}.
In particular, the same conclusion holds for both $\{ \nu_N^a, N \geq 1\}$ and $\{ \mu_N^a, N \geq 1\}$ and this shows that these two sequences converge and have the same limiting distribution. We are now going to show that the limiting measures can be expressed in terms of the occupation measure $\mu_\mathrm{occ}$ and a Poisson point process as explained in Theorem \ref{th: dim 3 convergence measures}. We start with the subcritical regime ($a<1$). Take $A_i \times T_i, i=1 \dots r$, as in Proposition \ref{prop: dim 3 moment number thick points}, $k_1, \dots, k_r \geq 1$ and denote $k=k_1 + \cdots + k_r$.
As
\[
(x,y) \mapsto a_d \left( \abs{x-y}^{2-d} - q(x,y) \right)
\]
is the Green function associated to Brownian motion killed at the first exit time $\tau$ of $[-1,1]^d$ (see equation (3.15) of \cite{bass1995} for instance), it is not hard to see that
\begin{align*}
& \EXPECT{0}{ \prod_{i = 1}^r \mu_\mathrm{occ}(A_i)^{k_i} } \\
& =
\sum_{\sigma \in \mathfrak{S}_k} \int_{A_1^{k_1} \times \cdots \times A_r^{k_r}} \prod_{i=0}^{k-1} a_d \left( \abs{y_{\sigma(i+1)} - y_{\sigma(i)}} - q \left( y_{\sigma(i)}, y_{\sigma(i+1)} \right) \right) dy_1 \cdots dy_k
\end{align*}
with the convention $y_{\sigma(0)} = 0$. Thus
\begin{equation}\label{eq:dim3 proof thm convergence measures}
\Expect{ \prod_{i=1}^r \left( \frac{1}{g} \mu_\mathrm{occ} (A_i) \int_{T_i} e^{-t_i/g} \frac{dt_i}{g} \right)^{k_i} } = m(A_i \times T_i, k_i, i=1 \dots r).
\end{equation}
This proves the identification \eqref{eq:thm measure subcritical} of the limiting measure in the subcritical regime.
Let us now consider the critical case $a=1$. Recalling the definition of $f$ in \eqref{eq:dim3 def f} we see that the equation \eqref{eq:lem Poisson} of Lemma \ref{lem: dim 3 gamma laws} implies that if $P_1(\lambda_1), \dots, P_r(\lambda_r)$ are independent Poisson random variables with parameters $\lambda_1, \dots, \lambda_r$, 
\[
\Expect{P_1(\lambda_1)^{k_1} \dots P_r(\lambda_r)^{k_r}}
= \sum_{\substack{1 \leq q_i \leq k_i \\ i = 1 \dots r}} \left( \prod_{i=1}^r f(k_i \to q_i) \right) \lambda_1^{q_1} \dots \lambda_r^{q_r}.
\]
Using \eqref{eq:dim3 proof thm convergence measures}, this now shows \eqref{eq:thm measure critical} and it concludes the proof.
\end{proof}

We now move on to the proof of Theorem \ref{th: dim 3 convergence thick points}.

\begin{proof}[Proof of Theorem \ref{th: dim 3 convergence thick points}]
Take $a \in [0,1]$. In the proof of Theorem \ref{th: dim 3 convergence measures} we showed that
\[
\abs{\Mc_N(a)} / N^{2(1-a)} = \nu_N^a([-1,1]^d \times (0,\infty))
\]
converges to $\nu^a([-1,1]^d \times (0, \infty))$. The identities \eqref{eq:thm law number thick points subcritical} and \eqref{eq:thm law number thick points critical} come from \eqref{eq:thm measure subcritical} and \eqref{eq:thm measure critical} and from the fact that $\mu_\mathrm{occ}([-1,1]^d) = \tau$ a.s.
\end{proof}

We will finish this section by proving Theorem \ref{th: dim 3 convergence supremum}.

\begin{proof}[Proof of Theorem \ref{th: dim 3 convergence supremum}]
Let $t \in \R$. Because the discrete random variables
\[
\nu_N^1 \left( [-1,1]^d \times (t, \infty) \right), N \geq 1,
\]
converge in law to a Poisson distribution with parameter $\tau e^{-t/g} / g$, we have
\begin{align*}
\lim_{N \rightarrow \infty} \PROB{0}{ \sup_{x \in V_N} \ell_x^{\tau_N} - 2g \log N \leq t}
& = \lim_{N \rightarrow \infty} \PROB{0}{ \nu_N^1 \left( [-1,1]^d \times (t, \infty) \right) = 0 } \\
& = \Expect{ \exp \left( - \frac{\tau}{g} e^{-t/g} \right) }.
\end{align*}
This concludes the proof.
\end{proof}

\subsection{Proof of Proposition \ref{prop: dim 3 moment number thick points}}\label{sec: dim 3 proof proposition}

In this section, we will prove Proposition \ref{prop: dim 3 moment number thick points} stated in the previous section.
We are first going to lay the groundwork by stating some technical lemmas which will be used in the proof of Proposition \ref{prop: dim 3 moment number thick points}. These lemmas, except the next one, will be proven in Section \ref{sec: dim 3 proof lemmas}.

We start with a well-known and easy lemma that we state for convenience. This lemma is valid for more general Markov chains.

\begin{lemma}\label{lem:crucial local times hitting times}
For all subset $A \subset \Z^d$, starting from $x$, $\ell_x^{\tau_A}$ and $Y_{\tau_A} \indic{\tau_A < \infty}$ are independent.
\end{lemma}

\begin{proof}
Consider a trajectory of the random walk $Y$ starting at $x$ and killed at $\tau_A$. We can decompose it according to the excursions away from $x$. There is a geometric number of independent excursions. The last one is conditioned to not come back to $x$ whereas the previous ones are i.i.d. excursions conditioned to come back to $x$. To conclude the proof, we notice that $Y_{\tau_A} \indic{\tau_A < \infty}$ depends on the last excursion whereas $\ell_x^{\tau_A}$ depends on the previous ones.
\end{proof}

\begin{remark}
This lemma implies in particular that conditioned on $Y_{\tau_A} \indic{\tau_A < \infty}$ and starting from $x$, $\ell_x^{\tau_A}$ is still an exponential variable with mean $\EXPECT{x}{\ell_x^{\tau_A}}$.
We also want to emphasise that this lemma is no longer true if the walk does not start at $x$.
\end{remark}

Now, consider the $k$-th moment of $\nu_N^a \left( A \times T \right)$. To compute it, we will have to estimate the probability that in $k$ different points, say $x_1, \dots, x_k$, the local times belong to $2ga \log N + T$.
To capture the correlations of those local times, we will denote by $E$ (to ease notation, we omit the dependence in $N$ and $x_1, \dots, x_k$) the number of excursions between the $x_i$'s before the time $\tau_N$. More precisely, if we define
\begin{align*}
\varsigma_0 & := \inf \left\{ t \geq 0: Y_t \in \{ x_1, \dots, x_k \} \right\}, \\
\forall p \geq 1, \varsigma_p & := \inf \left\{ t \geq \varsigma_{p-1}: Y_t \in \{ x_1, \dots, x_k \} \backslash \left\{ Y_{\varsigma_{p-1}} \right\} \right\},
\end{align*}
then
\begin{equation}\label{eq: dim 3 def E}
E := \max \left\{ p \in \N, \varsigma_p \leq \tau_N \right\}
\end{equation}
with the convention $\max \varnothing = - \infty$. 
The lemma below studies some properties of $E$. It roughly states that the typical way to visit all the points $x_1, \dots, x_k$ corresponds to $E = k-1$. It means that there exists a permutation $\sigma$ of the set of indices $\{1, \dots, k\}$ so that we have the following: the walk first hits $x_{\sigma(1)}$, then hits $x_{\sigma(2)}$, etc. When the walk has visited $x_{\sigma(i)}$ it does not come back to the vertices $x_{\sigma(1)}, \dots, x_{\sigma(i-1)}$.
We will denote $\mathfrak{S}_k$ the set of permutations of $\{1,\dots,k\}$.

\begin{lemma}\label{lem: dim 3 number of excursions E}
There exist $C_k >0$ and an integrable function
\begin{equation}\label{eq: dim 3 def U in lemma}
U : \left\{ (y_1, \dots, y_k ) \in \left( [-1,1]^d \backslash \{ 0 \} \right)^k: \forall i \neq j, y_i \neq y_j \right\} \rightarrow (0, \infty)
\end{equation}
such that the following is true.
For all $(y_1, \dots, y_k)$ and $(y_1', \dots, y_k')$ where $U$ is defined we have
\begin{equation}\label{eq:lem U is regular}
U (y_1, \dots, y_k) \leq \max_{0 \leq i \neq j \leq k} \left( \frac{\abs{y_i'-y_j'}}{\abs{y_i-y_j}} \right)^{d-2} U(y'_1, \dots, y'_k)
\end{equation}
with the convention $y_0 = y_0' = 0$.
For all $p \geq k-1$ and all $x_1, \dots, x_k$ non zero and pairwise distinct elements of $V_N$,
\begin{equation}\label{eq:lem number excursions E upper bound}
\PROB{0}{E = p, \tau_{x_i} <\tau_N ~\forall i = 1 \dots k} \\
\leq
C_k^{p+1} \left( \max_{i \neq j} \abs{x_i - x_j}^{2-d} \right)^{p-k+1}
N^{(2-d)k} U \left( \tfrac{x_1}{N}, \dots, \tfrac{x_k}{N} \right) .
\end{equation}
Moreover, if $x_1 = \floor{N y_1}, \dots, x_k = \floor{Ny_k}$, for $y_1, \dots, y_k$ non zero and pairwise distinct elements of $(-1,1)^d$, we have the following pointwise estimate:
\begin{align}
\lim_{N \rightarrow \infty} N^{(d-2)k} & \PROB{0}{E = k-1, \tau_{x_i} <\tau_N ~\forall i = 1 \dots k} \nonumber \\
& =
\left( \frac{a_d}{g} \right)^k \sum_{\sigma \in \mathfrak{S}_k}
\prod_{i = 0}^{k-1} \left( \abs{y_{\sigma(i+1)}-y_{\sigma(i)}}^{2-d} - q \left( y_{\sigma(i)},y_{\sigma(i+1)} \right) \right) \label{eq:lem number excursions E asymptotic}
\end{align}
with the convention $y_{\sigma(0)} = 0$.
\end{lemma}

\begin{remark} It is important for us to give a better estimate than
\[
\forall p \geq k-1, \PROB{0}{E = p, \tau_{x_i} <\tau_N ~\forall i = 1 \dots k}
\leq 
C_k^p \max_i \abs{x_i}^{2-d}  \left( \max_{i \neq j} \abs{x_i - x_j}^{2-d} \right)^p
\]
because the function
\[
(y_1, \dots, y_k) \in \prod_{i=1}^k (-1,1)^d \mapsto \max_i \abs{y_i}^{2-d}  \left( \max_{i \neq j} \abs{y_i - y_j}^{2-d} \right)^{k-1} \in (0,\infty)
\]
is not integrable if $(k-1)(d-2) \geq d$.
\end{remark}

\bigbreak
As mentioned in Section \ref{sec: Literature Overview and Organization of the Paper}, in the subcritical regime we will be able to restrict ourselves to points $x_1, \dots, x_k$ which are far away from each other. At criticality we will have to deal with points which are close to each other. The following lemma shows that two distinct close points are not thick at the same time with high probability:

\begin{lemma}\label{lem: dim 3 two close points are not both thick}
For $x,y \in \Z^d$, consider a sequence $\left( \ell_x^{\infty, i} , \ell_y^{\infty,i} \right), i \geq 1$, of i.i.d. variables with the same law as $\left( \ell_x^\infty, \ell_y^\infty \right)$ under $\prob_x$. If $x \neq y$, then for all $p \geq 1$, there exists $\eps_p >0$ independent of $x$ and $y$ such that for all $t \in \R$,
\[
\Prob{ \sum_{i=1}^p \ell_x^{\infty,i} , \sum_{i=1}^p \ell_y^{\infty,i} \geq 2g\log N + gt } \leq N^{-2 - \eps_p + o(1)}.
\]
\end{lemma}

We have now all the ingredients we need to start the proof of Proposition \ref{prop: dim 3 moment number thick points}.

\begin{proof}[Proof of Proposition \ref{prop: dim 3 moment number thick points}]
To ease notations, we will restrict ourselves to the case of the $k$-th moment of $\nu_N^a \left( A \times T \right)$ for $A \in \Bc([-1,1]^d)$ such that the Lebesgue measure of $\bar{A}\backslash A^\circ$ vanishes and $T \in \Bc(\R)$ with $\inf T > - \infty$ ($\inf T >0$ if $a = 0$). Indeed, the proof of the general case follows almost entirely along the same lines and throughout the proof we will explain which arguments need to be changed to treat the case of mixed moments
\[
\EXPECT{0}{ \prod_{i=1}^r \left\{ \nu_N^a(A_i \times T_i) \right\}^{k_i} }.
\]
When we will refer to the general case, $k$ will denote $k_1 + \dots + k_r$.

In the following, we will take $N$ large enough so that $2ga \log N + T \subset (0,\infty)$. To ease notations, we will denote
\begin{equation}\label{eq: dim 3 proof prop def M_N and A_N}
M_N := \nu_N^a \left( A \times T \right)
\mathrm{~and~}
A_N := \{ x \in V_N: x/N \in A \}.
\end{equation}
The $k$-th moment of $M_N$ can be written as
\[
\EXPECT{0}{ \left( M_N \right)^k }
= N^{-2(1-a)k} \sum_{x_1, \dots, x_k \in A_N} \PROB{0}{\ell_{x_1}^{\tau_N}, \dots, \ell_{x_k}^{\tau_N} \in 2ga \log N + T}.
\]
For some $r_N = N^{o(1)}$ (to be chosen later on), we introduce the set of well-separated points
\[
A_{N,k} := \left\{ (x_1, \dots, x_k) \in \left( A_N \backslash \{0\} \right)^k: \min_{i \neq j} \abs{x_i - x_j} > 2r_N \right\}.
\]
The proof will be decomposed in four parts. The first one will estimate the contribution of $A_{N,k}$ to the $k$-th moment of $M_N$. This part does not need to treat the subcritical ($a<1$) and critical ($a=1$) cases separately. Then, the second part shows that the contribution of points $(x_1, \dots, x_k) \in (A_N)^k \backslash A_{N,k}$ to the $k$-th moment of $M_N$ vanishes in the subcritical regime. The third part deals with the critical case and handles the points that are close to each other. The fourth part will briefly show the results on $\mu_N^a$.

\paragraph{Contribution of points far away from each other, $\nu_N^a$.}\label{para:points_far_away}

The goal of this part is to show that for all $a \in [0,1]$,
\begin{equation}
\label{eq:proof_M_N,k}
\lim_{N \to \infty}
N^{-2(1-a)k} \sum_{(x_1, \dots, x_k) \in A_{N,k}} \PROB{0}{\ell_{x_1}^{\tau_N}, \dots, \ell_{x_k}^{\tau_N} \in 2ga \log N + T}
= m(A \times T,k).
\end{equation}
We will write
\[
M_{N,k} := N^{-2(1-a)k} \sum_{(x_1, \dots, x_k) \in A_{N,k}} \PROB{0}{\ell_{x_1}^{\tau_N}, \dots, \ell_{x_k}^{\tau_N} \in 2ga \log N + T}.
\]
For a given $x \in V_N \backslash \partial V_N$, the Lebesgue measure of the set $\{ y \in (-1,1)^d: \floor{Ny} = x \}$ is $(1/N)^d$. Hence we can write
\begin{align}\label{eq: k-th moment as an integral}
M_{N,k}
=
N^{(d-2+2a)k} \int_{\prod_{i=1}^k(-1,1)^d} & \PROB{0}{\ell_{\floor{Ny_1}}^{\tau_N}, \dots, \ell_{\floor{Ny_k}}^{\tau_N} \in 2ga \log N + T} \nonumber \\
& ~~~~\times \indic{ \left( \floor{Ny_1}, \dots, \floor{Ny_k} \right) \in A_{N,k}} dy_1 \dots dy_k.
\end{align}
We will first bound from above the integrand. This will provide us the domination we need in order to apply the dominated convergence theorem and we will be left to show the pointwise limit.

\bigbreak
Let $(x_1, \dots, x_k) \in A_{N,k}$. By definition of $E$ (equation \eqref{eq: dim 3 def E}), if the walk visits all the $x_i$'s before $\tau_N$, then $E \geq k-1$. Thus
\[
\PROB{0}{\ell_{x_1}^{\tau_N}, \dots, \ell_{x_k}^{\tau_N} \in 2ga \log N + T, E \leq k-2} = 0.
\]
In this paragraph, we will use Lemma \ref{lem: dim 3 number of excursions E} to show that the probability
\[
\PROB{0}{\ell_{x_1}^{\tau_N}, \dots, \ell_{x_k}^{\tau_N} \in 2ga \log N + T, E \geq k}
\]
is very small. First, by denoting $t := \inf T /g$, we can bound
\[
\PROB{0}{\ell_{x_1}^{\tau_N}, \dots, \ell_{x_k}^{\tau_N} \in 2ga \log N + T, E \geq k}
\leq \PROB{0}{\ell_{x_1}^{\tau_N}, \dots, \ell_{x_k}^{\tau_N} > 2ga \log N + gt, E \geq k}.
\]
Starting from $x_1$, the law of the time spent in $x_1$ before hitting $\partial V_N \cup \{x_2, \dots, x_k \}$ is an exponential law with mean at most $g$. Also, if $E = p$, the number of excursions from $x_1$ to $\{x_2, \dots, x_k \}$ before $\tau_N$ is not larger than $p$. Hence, by Lemma \ref{lem:crucial local times hitting times} conditioned on the event $\{ E = p, \tau_{x_i} < \tau_N ~\forall i \leq k \}$, the joint law $(\ell_{x_1}^{\tau_N}, \dots, \ell_{x_k}^{\tau_N})$ is stochastically dominated by the law of $k$ independent Gamma random variables with shape parameter $p+1$ and scale parameter $g$. Using the claim \eqref{eq:lem Gamma} of Lemma \ref{lem: dim 3 gamma laws} about the Gamma distribution, it implies that
\begin{align*}
& \PROB{0}{\ell_{x_1}^{\tau_N}, \dots, \ell_{x_k}^{\tau_N} > 2ga \log N + gt \lvert E = p, \tau_{x_i} < \tau_N ~\forall i \leq k} \\
& \leq N^{-2ak} e^{-kt} \sum_{q = 0}^{kp} (2a\log N + t)^q \frac{k^q}{q!}.
\end{align*}
By definition of $A_{N,k},$ $\min_{i \neq j} \abs{x_i - x_j} \geq 2 r_N$. Let $U(x_1, \dots, x_k)$ be as in Lemma \ref{lem: dim 3 number of excursions E}. Then
\begin{align}
& \PROB{0}{\ell_{x_1}^{\tau_N}, \dots, \ell_{x_k}^{\tau_N} > 2ga \log N + gt, E \geq k} \nonumber\\
& = \sum_{p \geq k} \PROB{0}{E = p, \tau_{x_i} < \tau_N ~\forall i \leq k} \nonumber \\
& ~~~~~~~~~~~\times \PROB{0}{\ell_{x_1}^{\tau_N}, \dots, \ell_{x_k}^{\tau_N} > 2ga \log N + gt \lvert E = p, \tau_{x_i} < \tau_N ~\forall i \leq k} \nonumber \\
& \leq N^{-(d-2+2a)k} e^{-kt} U \left( \frac{x_1}{N}, \dots, \frac{x_k}{N} \right) \sum_{p \geq k} \left( C_k r_N^{\frac{2-d}{k}} \right)^p \sum_{q = 0}^{kp} (2a\log N + t)^q \frac{k^q}{q!} \nonumber\\
& = N^{-(d-2+2a)k} e^{-kt} U \left( \frac{x_1}{N}, \dots, \frac{x_k}{N} \right) \sum_{q \geq 0} \frac{((2a\log N + t)k)^q}{q!} \sum_{p \geq \lceil q/k \rceil \vee k } \left( C_k r_N^{\frac{2-d}{k}} \right)^p \nonumber\\
& \leq C_k' N^{-(d-2+2a)k} e^{-kt} U \left( \frac{x_1}{N}, \dots, \frac{x_k}{N} \right) \sum_{q \geq 0} \frac{((2a\log N + t)k)^q}{q!} \left( C_k r_N^{\frac{2-d}{k}} \right)^{\lceil q/k \rceil \vee k} \nonumber\\
& \leq C_k'' r_N^{\frac{2-d}{2}} N^{-(d-2+2a)k} e^{-kt} U \left( \frac{x_1}{N}, \dots, \frac{x_k}{N} \right) \sum_{q \geq 0} \left\{ (2a\log N + t)k C_k^{\frac{1}{2k}} r_N^{\frac{2-d}{2k^2}} \right\}^q / q! \label{eq: dim 3 useful a=1 E >=k}
\end{align}
because $\left\lceil \frac{q}{k} \right\rceil \vee k \geq \frac{k}{2} + \frac{q}{2k}$ for all $q \geq 0$. If we choose $r_N = \exp \left( \sqrt{\log N} \right) =  N^{o(1)}$ for instance, then $(2a\log N + t)k C_k^{1/(2k)} r_N^{(2-d)/(2k^2)}$ goes to zero and we have obtained:
\begin{equation}\label{eq: term E >= k negligible}
\PROB{0}{\ell_{x_1}^{\tau_N}, \dots, \ell_{x_k}^{\tau_N} \geq 2ga \log N + gt, E \geq k}
\leq  o(1) N^{-(d-2+2a)k} e^{-kt} U \left( \frac{x_1}{N}, \dots, \frac{x_k}{N} \right).
\end{equation}
According to Lemma \ref{lem: dim 3 number of excursions E}, the function
$(y_1, \dots, y_k) \in (-1,1)^k \mapsto U(y_1, \dots, y_k) \in (0,\infty)$
is integrable. Moreover, the equation \eqref{eq:lem U is regular} of Lemma \ref{lem: dim 3 number of excursions E} implies that
if $y_1, \dots, y_k \in (-1,1)^d$ are such that $(\floor{Ny_1}, \dots, \floor{Ny_k}) \in A_{N,k}$, then
\[
U \left( \frac{\floor{Ny_1}}{N}, \dots, \frac{\floor{Ny_k}}{N} \right) \leq C_{k,d} U(y_1, \dots, y_k)
\]
for some $C_{k,d} >0$.
Coming back to the equation \eqref{eq: k-th moment as an integral} we have thus shown with the equation \eqref{eq: term E >= k negligible} that:
\begin{align}\label{eq: k-th moment with E = p-1}
M_{N,k}
=
o(1) +
N^{(d-2+2a)k} & \int_{\prod_{i=1}^k(-1,1)^d} dy_1 \dots dy_k \indic{ \left( \floor{Ny_1}, \dots, \floor{Ny_k} \right) \in A_{N,k}} \\
& ~~ \times \PROB{0}{\ell_{\floor{Ny_1}}^{\tau_N}, \dots, \ell_{\floor{Ny_k}}^{\tau_N} \in 2ga \log N + T, E = k-1}. \nonumber
\end{align}

\bigbreak
Our last task consists in controlling the probability appearing in the equation \eqref{eq: k-th moment with E = p-1}. By Lemma \ref{lem:crucial local times hitting times}, conditioning on the event $\{ E = k-1, \tau_{x_i} <\tau_N ~\forall i = 1 \dots k \}$, the local times $\ell_{x_i}^{\tau_N}, i = 1 \dots k$, are independent exponential variables with mean
$\EXPECT{x_i}{\ell_{x_i}^{\tau_N \wedge \min_{j \neq i} \tau_{x_j}} } \leq g$. Consequently,
\begin{align}
\prob_0 ( \ell_{x_1}^{\tau_N}, \dots, &\ell_{x_k}^{\tau_N} \in 2ga \log N + T, E = k-1 ) \nonumber\\
& \leq N^{-2ak} \left( \int_T \frac{1}{g} e^{-s/g} ds \right)^k \PROB{0}{E = k-1, \tau_{x_i} <\tau_N ~\forall i \leq k}. \label{eq: dim 3 useful for a=1 E =k-1}
\end{align}
Using the first estimate of Lemma \ref{lem: dim 3 number of excursions E}, it implies that $ M_{N,k}$ is bounded and it also provides us the domination we need to use the dominated convergence theorem.
We have already done everything we need for the pointwise convergence. Indeed, if $x_1 = \floor{N y_1}, \dots, x_k = \floor{Ny_k}$, for $y_1, \dots, y_k$ non zero and pairwise distinct elements of $(-1,1)^d$, Lemma \ref{lem: dim 3 number of excursions E} provides an explicit expression for the pointwise limit
\[
\lim_{N \rightarrow \infty} N^{(d-2)k} \PROB{0}{E = k-1, \tau_{x_i} <\tau_N ~\forall i = 1 \dots k}
\]
and a small modification of the arguments in the proof of Lemma \ref{lem: dim 3 number of excursions E} shows that
\[
\EXPECT{\floor{Ny_i}}{\ell_{\floor{Ny_i}}^{\tau_N \wedge \min_{j \neq i} \tau_{\floor{Ny_j}}} } = g + O_{y_1, \dots, y_k} \left( N^{2-d} \right).
\]
Hence
\begin{align*}
& \lim_{N \rightarrow \infty} N^{2ka}
\PROB{0}{\ell_{x_1}^{\tau_N}, \dots, \ell_{x_k}^{\tau_N} \in 2ga \log N + T \lvert E = k-1, \tau_{x_i} <\tau_N ~\forall i = 1 \dots k} \\
& = \left( \int_T e^{-s/g} \tfrac{ds}{g} \right)^k.
\end{align*}
Moreover,
\begin{align*}
\indic{\forall i \neq j, y_i \in A^\circ \backslash \{0\}, y_i \neq y_j }
& \leq
\liminf_{N \rightarrow \infty} \indic{ \left( \floor{Ny_1}, \dots, \floor{Ny_k} \right) \in A_{N,k}} \\
& \leq
\limsup_{N \rightarrow \infty} \indic{ \left( \floor{Ny_1}, \dots, \floor{Ny_k} \right) \in A_{N,k}} \leq \indic{\forall i \neq j, y_i \in \bar{A} \backslash \{0\}, y_i \neq y_j }.
\end{align*}
Notice the interior $A^\circ$ and the closure $\bar{A}$ in the previous inequalities.
As we have supposed that the Lebesgue measure of $\bar{A} \backslash A^\circ$ vanishes, putting things together leads to the convergence of $M_{N,k}$ to
\begin{align*}
\left( \frac{a_d}{g} \right)^k \left( \int_T e^{-s/g} \frac{ds}{g} \right)^k \sum_{\sigma \in \mathfrak{S}_k} & \int_{A^k}
\times \prod_{i=0}^{k-1} \left( \abs{y_{\sigma(i+1)} - y_{\sigma(i)}}^{2-d} - q(y_{\sigma(i)}, y_{\sigma(i+1)}) \right) dy_1 \dots dy_k
\end{align*}
with the convention $y_{\sigma(0)} = 0$. This completes the proof of \eqref{eq:proof_M_N,k}.

\paragraph{Subcritical regime, $\nu_N^a$.}

We now show how the previous part allows us to conclude the proof in the subcritical regime. Suppose that $a<1$. We show that the $k$-th moment of $M_N$ converges towards $m(A \times T,k)$ by induction on $k \geq 1$.
Thanks to \eqref{eq:proof_M_N,k}, it only remains to control the contribution of points $(x_1, \dots, x_k) \in (A_N)^k \backslash A_{N,k}$ to the $k$-th moment of $M_N$. This contribution is at most
\begin{align*}
C(k,d) N^{-2(1-a)k} r_N^d & \sum_{x_1, \dots, x_{k-1} \in A_N} \PROB{0}{\ell_{x_1}^{\tau_N}, \dots, \ell_{x_{k-1}}^{\tau_N} \in 2ga \log N + T} \\
& = C(k,d) N^{-2(1-a)} r_N^d \EXPECT{0}{ \left( M_N \right)^{k-1} }
\end{align*}
which goes to zero: this is clear for $k=1$ (because $r_N = N^{o(1)}$ and $a < 1$) and comes from the induction hypothesis for $k \geq 2$. With \eqref{eq:proof_M_N,k}, we have shown that
\[
\EXPECT{0}{ \left( M_N \right)^k } = m(A \times T,k) + o(1).
\]
This is exactly \eqref{eq:prop subcritical regime} in the case $r=1$. In the general case of a mixed moment, we recover the result by the exact same method.

\paragraph{At criticality, $\nu_N^a$.}
Let us now consider the critical case $a =1$.
Unlike in the subcritical regime, the points $(x_1, \dots, x_k) \in (A_N)^k \backslash A_{N,k}$ will contribute to $\EXPECT{0}{(M_N)^k}$. We first notice that the points $(x_1, \dots, x_k) \in (A_N)^k$ with one of the $x_i$'s being equal to zero do not contribute. Indeed, by ignoring the points which are within a distance $2 r_N$ to each other or to zero, which contributes at most $C r_N^d$ for every such point, we have:
\begin{align*}
\sum_{\substack{(x_1, \dots, x_k) \in (A_N)^k \\ \exists i, x_i = 0}}
& \PROB{0}{ \ell_{x_1}^{\tau_N}, \dots, \ell_{x_k}^{\tau_N} \in 2g \log N + T} \\
& \leq
C_k \sum_{l = 0}^{k-1} \left( C r_N^d \right)^{k-1-l} \sum_{\substack{\forall i = 1 \dots l, \abs{x_i} \geq 2r_N \\ \forall i \neq j, \abs{x_i-x_j} \geq 2r_N}}
\PROB{0}{\ell_0^{\tau_N}, \ell_{x_1}^{\tau_N}, \dots, \ell_{x_l}^{\tau_N} \in 2g \log N + T}.
\end{align*}
The last sum is over $l$ different points and we require the local times to be large in $l+1$ different points.
We can then use the same arguments as in Section \ref{para:points_far_away} (all the points are far away from each other) to show that this last sum is at most $CN^{-2}$. As $r_N = N^{o(1)}$ it shows that this contribution vanishes.

We are going to estimate
\begin{equation}\label{eq:proof sum close points}
\sum_{\substack{(x_1, \dots, x_k) \in (A_N \backslash \{0\})^k \backslash A_{N,k} }}
\PROB{0}{ \ell_{x_1}^{\tau_N}, \dots, \ell_{x_k}^{\tau_N} \in 2g \log N + T}.
\end{equation}
If $(x_1, \dots, x_k) \in (A_N \backslash \{0\})^k \backslash A_{N,k}$, by definition of $A_{N,k}$, it means that there are at least two balls $B(x_i, r_N)$ which overlap. In the following, we will partition the set $(A_N \backslash \{0\})^k \backslash A_{N,k}$ according to the maximum number $r$ ($r \leq k-1$) of balls which do not overlap. We will denote by $x_{i_p}$, $p=1 \dots r$, the centres of such balls and we will partition the set of indices $\sqcup_{p=1}^r I_p = \{1, \dots, k\}$ such that for all $p=1 \dots r, i \in I_p, \abs{x_i - x_{i_p}} \leq 2 r_N$. See Figure \ref{fig: decomposition}. The reader should think of the balls as small balls which are far away from each other. The choice of the partition $(I_p)$ may be not unique. In this case, we make an arbitrary choice.

\begin{figure}
\centering
\begin{tikzpicture}[scale=1]

\node (A) at (2,0.8) {};
\node (A1) at (3,0.8) {};
\node (A2) at (2,2.8) {};
\node (B) at (10,0) {};
\node (B1) at (10.7,0.2) {};
\node (B2) at (8.7,-0.8) {};
\node (B3) at (10.2,1.8) {};
\node (text) at (7,0) {} ;

\fill (B1) circle (0.05);
\fill (B2) circle (0.05);
\fill (B3) circle (0.05);

\draw (A) node [left] {$x_{i_1}$};
\draw (B) node [below] {$x_{i_2}$};
\draw (text) node [left] {$x_i, i \in I_2 \backslash \{ i_2 \}$};
\fill (A) circle (0.05);
\fill (B) circle (0.05);
\draw (A) circle (1);
\draw (B) circle (1);
\draw[dotted] (A) circle (2);
\draw[dotted] (B) circle (2);

\draw [->] (A) edge node [below] {$r_N$} (A1);
\draw [->] (A) edge node [above right] {$2r_N$} (A2);
\draw [->] (text) edge (B1);
\draw [->] (text) edge (B2);
\draw [->] (text) edge (B3);
\end{tikzpicture}
\caption{Decomposition of $(A_N \backslash \{0\})^k \backslash A_{N,k}$. The balls in solid lines do not overlap. Here $r=2$.}\label{fig: decomposition}
\end{figure}

Our decomposition is thus:
\begin{align*}
(A_N \backslash \{0\})^k \backslash A_{N,k} = \bigcup_{r = 1}^{k-1} \bigcup_{\substack{\sqcup_{p=1}^r I_p \\ = \{1,\dots,k\} }}
W_{N,k,r,(I_p)}
\end{align*}
where
\[
W_{N,k,r,(I_p)} = \left\{
(x_1, \dots, x_k) \in (A_N \backslash \{0\})^k:
\begin{array}{l}
\forall p \neq q, \exists i_p \in I_p, i_q \in I_q, \abs{x_{i_p} - x_{i_q}} > 2r_N, \\
\forall i \in I_p, \abs{x_i - x_{i_p}} \leq 2r_N
\end{array}
\right\}.
\]
For a given $W_{N,k,r,(I_p)}$, the contribution to the sum \eqref{eq:proof sum close points} of the elements $(x_1, \dots, x_k) \in W_{N,k,r,(I_p)}$ such that for all $p=1 \dots r$, for all $i,j \in I_p, x_i = x_j$ is equal to
\[
\sum_{(y_1, \dots, y_r) \in A_{N,r}} \PROB{0}{\ell_{y_1}^{\tau_N}, \dots, \ell_{y_r}^{\tau_N} \in 2g \log N + T}
\]
which converges to $m(A \times T,r)$ (see \eqref{eq:proof_M_N,k}).
As the number of ways to partition the set $\{1, \dots, k\}$ into $r$ non empty sets is exactly equal to $f(k \to r)$, the claim of the proposition is equivalent to saying that the contribution of $W_{N,k,r,(I_p)}$ to the sum \eqref{eq:proof sum close points} comes only from these points. In other words, if we denote
\[
W_{N,k,r,(I_p)}^{\neq} = \left\{ (x_1, \dots, x_k) \in W_{N,k,r,(I_p)}: \exists p = 1 \dots r, \exists i,j \in I_p, x_i \neq x_j \right\}
\]
then we are going to show that
\[
\sum_{(x_1, \dots, x_k) \in W_{N,k,r,(I_p)}^{\neq}} \PROB{0}{\ell_{x_1}^{\tau_N}, \dots, \ell_{x_k}^{\tau_N} \in 2g \log N + T} \xrightarrow[N \rightarrow \infty]{} 0.
\]
By denoting $t := \inf T / g$, we can first bound:
\[
\PROB{0}{\ell_{x_1}^{\tau_N}, \dots, \ell_{x_k}^{\tau_N} \in 2g \log N + T}
\leq
\PROB{0}{\ell_{x_1}^\infty, \dots, \ell_{x_k}^\infty > 2g \log N + gt}.
\]
If $(x_1, \dots, x_k) \in W_{N,k,r,(I_p)}^{\neq}$, then there exists $p_0 \in \{1, \dots, r\}$ and $j_{p_0} \in I_{p_0}$ such that $x_{i_{p_0}} \neq x_{j_{p_0}}$. To bound from above this last sum, for each $p \neq p_0$ we keep track of only one $x_k, k \in I_p,$ by considering $x_{i_p}$. As for all $k \in I_p$, $\abs{x_k - x_{i_p}} \leq 2r_N$, our estimate is increased by a multiplicative factor of order $r_N^d$ for each point that we forget.
For $p=p_0$, we keep track of both $x_{i_{p_0}}$ and $x_{j_{p_0}}$. Furthermore, $x_{j_{p_0}}$ will absorb all the $x_{i_p}, p \neq p_0$ which are within a distance $2r_N$ of $x_{j_{p_0}}$. This procedure implies that:
\begin{align}\label{eq: dim 3 a=1 max s}
& \sum_{(x_1, \dots, x_k) \in W_{N,k,r,(I_p)}^{\neq}} \PROB{0}{\ell_{x_1}^\infty, \dots, \ell_{x_k}^\infty > 2g \log N + gt} \\
& ~~~~~~~~~~~~~~~~~~ \leq
C \sum_{s = 1}^r (r_N^d)^{k-s-1}
\sum_{\substack{x_0, \dots, x_s \in A_N \\ x_0 \neq x_1, \abs{x_0 - x_1} \leq 2r_N\\ \forall i \neq j, \{i,j\} \neq \{0,1\}, \abs{x_i - x_j} > 2r_N}}
\PROB{0}{\ell_{x_0}^\infty, \dots, \ell_{x_s}^\infty > 2g \log N + gt} \nonumber
\end{align}
where $C>0$ may depend on $d,k,r$. We will conclude by showing that this last sum is not larger than $N^{-\eps}$ for some $\eps>0$. Take $s \in \{1, \dots, r\}$ and $(x_0, x_1, \dots, x_s)$ as in the previous sum. If $s=1$ it means that we just need to control the local times $\ell_{x_0}^\infty, \ell_{x_1}^\infty$. This has already been done in Lemma \ref{lem: dim 3 two close points are not both thick} and we are going to explain the slightly more delicate case $s \geq 2$.
The idea is fairly similar to the one we used in the subcritical regime. Let us denote $E$ the number of excursions between the sets $\{x_0,x_1\}, \{x_2\}, \dots, \{x_s\}$. First of all, let us notice that if we take $p_{\max}\geq s$, a small modification of the equation \eqref{eq: dim 3 useful a=1 E >=k} gives:
\begin{align*}
& \sum_{\substack{x_0, \dots, x_s \in A_N \\ x_0 \neq x_1, \abs{x_0 - x_1} \leq 2r_N\\ \forall i \neq j, \{i,j\} \neq \{0,1\}, \abs{x_i - x_j} > 2r_N}}
\PROB{0}{\ell_{x_0}^\infty, \dots, \ell_{x_s}^\infty > 2g \log N + gt, E \geq p_{\max}} \\
& \leq
C(s,d) (r_N)^d
\sum_{\substack{x_1, \dots, x_s \in A_N \\ \forall i \neq j, \abs{x_i - x_j} > 2r_N}}
\PROB{0}{\ell_{x_1}^\infty, \dots, \ell_{x_s}^\infty > 2g \log N + gt, E \geq p_{\max}} \\
& \leq
C e^{-st} r_N^{(p_{\max}-s)(2-d)+d} N^{-ds}
\sum_{\substack{x_1, \dots, x_s \in A_N \\ \forall i \neq j, \abs{x_i - x_j} > 2r_N}} U \left(  \frac{x_1}{N}, \dots, \frac{x_s}{N} \right)
\leq C e^{-st} r_N^{(p_{\max}-s)(2-d)+d}.
\end{align*}
Hence if $p_{\max}$ is large enough, the negative power $(p_{\max}-s)(2-d) + d$ of $r_N$ will kill the positive power $(k-s-1)d$ of $r_N$ in the equation \eqref{eq: dim 3 a=1 max s} and we are now left to control:
\[
\sum_{\substack{x_0, \dots, x_s \in A_N \\ x_0 \neq x_1, \abs{x_0 - x_1} \leq 2r_N\\ \forall i \neq j, \{i,j\} \neq \{0,1\}, \abs{x_i - x_j} > 2r_N}}
\PROB{0}{\ell_{x_0}^\infty, \dots, \ell_{x_s}^\infty > 2g \log N + gt, E < p_{\max}}.
\]

Thanks to Lemmas \ref{lem: dim 3 gamma laws} and \ref{lem: dim 3 two close points are not both thick} and using the notations in those lemmas, we have
\begin{align*}
& \PROB{0}{\ell_{x_0}^\infty, \dots, \ell_{x_s}^\infty > 2g \log N + gt \lvert E = p, \tau_{ \{ x_0,x_1\} }, \tau_{x_2}, \dots, \tau_{x_s} < \infty} \\
& \leq
\Prob{\Gamma(p+1,g) > 2g \log N + gt}^{s-1}
\Prob{ \forall \alpha = 0,1, \sum_{i=1}^{p+1} \sum_{j=1}^{A_i} \ell_{x_\alpha,j}^i > 2g \log N + gt} \\
& \leq N^{-2s - \eps_p}.
\end{align*}
By summing \eqref{eq:lem number excursions E upper bound} of Lemma \ref{lem: dim 3 number of excursions E} over all $p \geq s-1$, we also have
\begin{align*}
\PROB{0}{E = p, \tau_{ \{ x_0,x_1\} }, \tau_{x_2}, \dots, \tau_{x_s} < \infty}
& \leq
2 \max_{\alpha = 0,1} \PROB{0}{\tau_{ x_\alpha }, \tau_{x_2}, \dots, \tau_{x_s} < \infty} \\
& \leq C N^{(2-d)s} \max_{\alpha = 0,1} U \left( \frac{x_\alpha}{N}, \frac{x_2}{N}, \dots, \frac{x_s}{N} \right).
\end{align*}
We have obtained the existence of $\eps > 0$ such that
\begin{align*}
& \sum_{\substack{x_0, \dots, x_s \in A_N \\ x_0 \neq x_1, \abs{x_0 - x_1} \leq 2r_N\\ \forall i \neq j, \{i,j\} \neq \{0,1\}, \abs{x_i - x_j} \geq 2r_N}}
\PROB{0}{\ell_{x_0}^\infty, \dots, \ell_{x_s}^\infty > 2g \log N + gt, E < p_{\max}} \\
& \leq
N^{-ds - \eps} \sum_{\substack{x_0, \dots, x_s \in A_N \\ x_0 \neq x_1, \abs{x_0 - x_1} \leq 2r_N\\ \forall i \neq j, \{i,j\} \neq \{0,1\}, \abs{x_i - x_j} \geq 2r_N}}
\max_{\alpha = 0,1} U \left( \frac{x_\alpha}{N}, \frac{x_2}{N}, \dots, \frac{x_s}{N} \right) \\
& \leq
C(d)(r_N)^d N^{-ds - \eps} \sum_{\substack{x_1, \dots, x_s \in A_N \\ \forall i \neq j, \abs{x_i - x_j} \geq 2r_N}} U \left( \frac{x_1}{N}, \frac{x_2}{N}, \dots, \frac{x_s}{N} \right) \leq C (r_N)^d N^{-\eps}
\end{align*}
where we justify as before the last inequality thanks to the integrability of $U$ and by \eqref{eq:lem U is regular}.
This concludes the proof of the estimates on $\EXPECT{0}{ \left\{\nu_N^a(A \times T) \right\}^k} $ at criticality (equation \eqref{eq:prop criticality} with $r=1$).

In the general case of a mixed moment, we have to deal with points
\begin{equation*}
\left\{ (x_1, \dots, x_k) \in (A_{1N} \backslash \{0\})^{k_1} \times \dots \times (A_{rN} \backslash \{0\})^{k_r}: \exists i \neq j, \abs{x_i - x_j} \leq 2r_N \right\}.
\end{equation*}
As before, we decompose this set according to blocks of points which are close to each other. Again, only points which are equal inside a same block will contribute. As we have assumed that the $A_i \times T_i$'s are pairwise disjoint, they will not interact between each other meaning that if $1 \leq i \neq j \leq r$, if $x_i \in A_i$ and $x_j \in A_j$, either $x_i \neq x_j$ or $T_i \cap T_j = \varnothing$.
Now, take $r_i \leq k_i$ for $i=1 \dots r$.
We notice that the number of ways to partition the sets $\{1, \dots, k_i\}$ into $r_i$ non empty sets, for $i=1 \dots r$, is equal to
\[
\prod_{i=1}^r f(k_i \to r_i).
\]
Thus, the contribution of points $(x_1, \dots, x_k) \in (A_{1N} \backslash \{0\})^{k_1} \times \dots \times (A_{rN} \backslash \{0\})^{k_r}$ such that for all $i = 1 \dots r$, $\{x_{k_1 + \dots + k_{i-1} + 1}, \dots, x_{k_1 + \dots + k_i} \}$ is composed of $r_i$ well-separated points converges to
\[
\left( \prod_{i=1}^r f(k_i \to r_i) \right) m(A_i \times T_i, r_i, i=1 \dots r).
\]
This shows \eqref{eq:prop criticality} in the general case $r \geq 1$.


\paragraph{Estimates on $\mu_N^a$.} We now briefly end the proof of Proposition \ref{prop: dim 3 moment number thick points} by explaining how the results for $\mu_N^a$ are obtained. Take $a \in [0,1]$, $T \in \Bc(\R)$ and $A \subset [-1,1]^d$ such that the Lebesgue measure of $\bar{A}\backslash A^\circ$ vanishes. By definition of $f(k \to r)$ and since $(E_x)_{x \in V_N}$ are i.i.d. exponential variables with mean $g$ independent of $\Mc_N(0)$, the normalised $k$-th moment
$\EXPECT{0}{( \mu_N^a (A \times T) )^k}$ is equal to
\begin{align*}
& \frac{1}{N^{2(1-a)k}} \EXPECT{0}{ \sum_{x_1, \dots, x_k \in A_N \cap \Mc_N(0)} \indic{E_{x_1}, \dots, E_{x_k} \in 2ga \log N + T} } \\
& = \frac{1}{N^{2(1-a)k}} \sum_{r=1}^k f(k \to r) \EXPECT{0}{ \sum_{\substack{x_1, \dots, x_r \in A_N \cap \Mc_N(0) \\\forall i \neq j, x_i \neq x_j}} \indic{E_{x_1}, \dots, E_{x_r} \in 2ga \log N + T} } \\
& = \frac{1}{N^{2(1-a)k}} \sum_{r=1}^k f(k \to r) N^{-2ar} \left( \int_T e^{-s/g} \frac{ds}{g} \right)^r \EXPECT{0}{\sum_{\substack{x_1, \dots, x_r \in A_N \\\forall i \neq j, x_i \neq x_j}} \indic{\ell_{x_1}^{\tau_N}, \dots, \ell_{x_r}^{\tau_N} > 0 } }.
\end{align*}
We have already shown that
\begin{align*}
\lim_{N \rightarrow \infty}
\frac{1}{N^{2r}}\EXPECT{0}{\sum_{\substack{x_1, \dots, x_r \in A_N \\ x_i \neq x_j \forall i \neq j}} \indic{\ell_{x_1}^{\tau_N}, \dots, \ell_{x_r}^{\tau_N} > 0 } }
= m(A \times (0, \infty),r)
\end{align*}
so $\EXPECT{0}{ \left( \mu_N^a (A \times T) \right)^k }$ converges to
\begin{align*}
\sum_{r=1}^k f(k \to r) \left( \int_T e^{-s/g} \frac{ds}{g} \right)^r m(A \times (0,\infty),r) \times
\left\{
\begin{array}{ll}
1 & \mathrm{if~} a=1 \mathrm{~or~} r = k \\
0 & \mathrm{if~} a<1 \mathrm{~and~} r<k
\end{array}
\right.
\end{align*}
which is exactly the stated result. The extension to the general case of a mixed moment is obtained exactly as for $\nu_N^a$.
\end{proof}

\subsection{Proof of technical lemmas}\label{sec: dim 3 proof lemmas}

We start this section by proving Lemma \ref{lem: behaviour Green function dimension >= 3} which gives estimates on the Green function $G_N$ (defined in \eqref{eq: dim 3 def Green functions} as well as the Green function $G$ on $\Z^d$) in dimension greater of equal to $3$.

\begin{proof}[Proof of Lemma \ref{lem: behaviour Green function dimension >= 3}]
As in dimension 2, these estimates follow from \cite{lawler1996intersections} and \cite{lawler_limic_2010}: Proposition 1.5.8 in \cite{lawler1996intersections} gives
\begin{equation}\label{eq:proof Green 1}
G_N(x,y) = G(x,y) - \sum_{z \in \partial V_N} \PROB{x}{Y_{\tau_N} = z} G(z,y)
\end{equation}
and Theorem 4.3.1 in \cite{lawler_limic_2010} (or Theorem 1.5.4 in \cite{lawler1996intersections} for a slightly worse estimate) gives
\begin{equation}\label{eq:proof Green 2}
G(x,y) = a_d \abs{x-y}^{2-d} + O \left( \abs{x-y}^{-d} \right) \mathrm{~as~} \abs{x-y} \rightarrow \infty.
\end{equation}
Our two first estimates on the Green function on the diagonal follow since if $y \in V_{(1-\eta)N}$ for some $\eta > 0$, then for all $z \in \partial V_N$, $\abs{z-y} \geq \eta N$. The lower bound on $q_N(x,y)$ follows as well. We are going to explain how to obtain the pointwise limit estimate \eqref{eq:lem pointwise limit Green}. Take $\tilde{x} \neq \tilde{y} \in (-1,1)^d$. By \eqref{eq:proof Green 1} and \eqref{eq:proof Green 2}, we have
\begin{align*}
N^{d-2} G_N \left( \floor{N \tilde{x}}, \floor{N \tilde{y}} \right)
& = a_d \abs{x-y}^{2-d} - a_d \EXPECT{\floor{N \tilde{x}}}{ \abs{ \frac{Y_{\tau_N}}{N} - \tilde{y} }^{2-d} } + O_{\tilde{x},\tilde{y}} \left( N^{2-d} \right).
\end{align*}
By Donsker's invariance principle, starting from $\floor{N \tilde{x}}$, $Y_{\tau_N}/N$ converges in law to the exit distribution of $[-1,1]^d$ of Brownian motion starting from $\tilde{x}$. We thus obtain \eqref{eq:lem pointwise limit Green}.
\end{proof}

We now move on to the proof of Lemma \ref{lem: dim 3 number of excursions E}. We consider $k$ non zero and pairwise distinct points $x_1, \dots, x_k \in V_N$ and we recall the definitions of $E$ and of the stopping times $\varsigma_p$ in \eqref{eq: dim 3 def E}.

\begin{proof}[Proof of Lemma \ref{lem: dim 3 number of excursions E}]
As mentioned just before Lemma \ref{lem: dim 3 number of excursions E},
if $E = k-1$ and $\tau_{x_i} <\tau_N ~\forall i = 1 \dots k$ then the stopping times $\varsigma_p, p=0 \dots k-1,$ define a permutation $\sigma$ of the set of indices $\{1, \dots, k \}$ which keeps track of the order of visits of the set $\{ x_1, \dots, x_k \}$.
By a repeated application of Markov property, we thus have:
\begin{align}\label{eq: dim 3 proof lemma number}
& \PROB{0}{E = k-1, \tau_{x_i} <\tau_N ~\forall i = 1 \dots k}
= \sum_{ \sigma \in \mathfrak{S}_k } \PROB{0}{\tau_{x_{\sigma(1)}} < \tau_N \wedge \min_{j \neq 1} \tau_{x_{\sigma(j)}} } \\
& ~~~~~~~~~~~~~~~~~~~~~~~~~~~~~~~~\times \prod_{i=1}^{k-1} \PROB{x_{\sigma(i)}}{ \tau_{x_{\sigma(i+1)}} < \tau_N \wedge \min_{j \neq i,i+1 } \tau_{x_{\sigma(j)}} }
\PROB{x_{\sigma(k)}}{ \tau_N < \min_{j \neq k} \tau_{x_{\sigma(j)}} }. \nonumber
\end{align}
But for all $\sigma \in \mathfrak{S}_k$ and $i = 1 \dots k-1$,
\[
\PROB{x_{\sigma(i)}}{ \tau_{x_{\sigma(i+1)}} < \tau_N \wedge \min_{j \neq i,i+1 } \tau_{x_{\sigma(j)}} }
\leq 
\PROB{x_{\sigma(i)}}{ \tau_{x_{\sigma(i+1)}} < \tau_N }
= \frac{ G_N(x_{\sigma(i)}, x_{\sigma(i+1)}) }{ G_N(x_{\sigma(i+1)}, x_{\sigma(i+1)}) }.
\]
We bound from below the denominator $G_N(x_{\sigma(i+1)}, x_{\sigma(i+1)})$ by 1 and from above the numerator $G_N(x_{\sigma(i)}, x_{\sigma(i+1)})$ by $C \abs{ x_{\sigma(i)} - x_{\sigma(i+1)} }^{2-d}$ (see Lemma \ref{lem: behaviour Green function dimension >= 3}). Coming back to \eqref{eq: dim 3 proof lemma number}, this leads to
\[
\PROB{0}{E = k-1, \tau_{x_i} <\tau_N ~\forall i = 1 \dots k}
\leq C^k \sum_{\sigma \in \mathfrak{S}_k} \prod_{i=0}^{k-1} \abs{ x_{\sigma(i)} - x_{\sigma(i+1)} }^{2-d}.
\]
with the convention $x_{\sigma(0)} = 0$.

The general case $p \geq k-1$ follows from the same lines but now the order of visits of the set $\{ x_1, \dots, x_k \}$ is not as simple as before. In the following, $\sigma \in \mathfrak{S}_k$ will keep track of the order of new visits of the vertices $x_1, \dots, x_k$: $x_{\sigma(1)}$ is the first vertex visited among the $x_i$'s, $x_{\sigma(2)}$ the second one... We will focus on the transitions which explore new vertices, so we introduce the notion:
$(\sigma, f) \in \mathfrak{S}_k \times \{1, \dots, k\}^{ \{2, \dots, k \} }$ is said to be admissible if
\[
\forall i = 2 \dots k, f(i) \in \{ \sigma(1), \dots, \sigma(i-1) \}.
\]
$x_{f(i)}$ will denote the vertex visited just before visiting the vertex $x_{\sigma(i)}$.
Now we define
\begin{equation}\label{eq: dim 3 def U}
U(x_1, \dots, x_k) := \sum_{(\sigma,f) \mathrm{~admissible}} \abs{x_{\sigma(1)}}^{2-d} \prod_{i=1}^{k-1} \abs{x_{\sigma(i+1)} - x_{f(i+1)}}^{2-d}.
\end{equation}
By keeping track of the transitions where new vertices are discovered (in a chronological sense) and by noticing that all the others occur with a probability which is not larger than $C_k \max_{i \neq j} \abs{x_i - x_j}^{2-d}$, we have
\begin{align*}
\PROB{0}{E = p, \tau_{x_i} <\tau_N ~\forall i = 1 \dots k}
& \leq (C_k)^{p+1} \left( \max_{i \neq j} \abs{x_i - x_j}^{2-d} \right)^{p-k+1} U(x_1, \dots, x_k) \\
= (C_k)^{p+1} & \left( \max_{i \neq j} \abs{x_i - x_j}^{2-d} \right)^{p-k+1}
N^{(2-d)k} U \left( \frac{x_1}{N}, \dots, \frac{x_k}{N} \right) .
\end{align*}
This proves \eqref{eq:lem number excursions E upper bound}.

We notice that \eqref{eq:lem U is regular} is immediate from the definition of $(y_1, \dots, y_k) \in (-1,1)^k \mapsto U(y_1, \dots, y_k) $ and we now check that it is integrable. Take $(\sigma,f)$ admissible. There is only one occurrence of $y_{\sigma(k)}$ in the product, so we can first integrate:
\[
\int_{(-1,1)^d} \abs{y_{\sigma(k)} - y_{f(k)}}^{2-d} dy_{\sigma(k)}
\leq \int_{(-2,2)^d+y_{f(k)}} \abs{y_{\sigma(k)} - y_{f(k)}}^{2-d} dy_{\sigma(k)} = C.
\]
We then proceed inductively by integrating next with respect to $y_{\sigma(k-1)}$, and so on. This proves that $U$ is integrable.

We now turn to \eqref{eq:lem number excursions E asymptotic}. If $x_1 = \floor{N y_1}, \dots, x_k = \floor{Ny_k}$, for $y_1, \dots, y_k$ non zero and pairwise distinct elements of $(-1,1)^d$, then there exists $\eta \in (0,1)$ such that for all $N$ large enough, $x_i \in V_{(1-\eta)N}, \abs{x_i} \geq \eta N$ and for all $i \neq j, \abs{x_i - x_j} \geq \eta N$. Hence Lemma \ref{lem: behaviour Green function dimension >= 3} implies
\begin{align*}
\PROB{x_1}{ \tau_{x_2} < \tau_N \wedge \min_{j \neq 1} \tau_{x_j} }
& = \PROB{x_1}{\tau_{x_2} < \tau_N} - \PROB{x_1}{ \exists j \neq 1, \tau_{x_j} < \tau_{x_2} < \tau_N} \\
& \geq \PROB{x_1}{\tau_{x_2} < \tau_N} - (k-2) \max_{j \neq 1} \PROB{x_1}{\tau_{x_j} < \tau_N} \PROB{x_j}{\tau_{x_2} < \tau_N} \\
& \geq \PROB{x_1}{\tau_{x_2} < \tau_N} - C_k (\eta N)^{2(2-d)}
\end{align*}
which leads to:
\begin{align*}
\lim_{N \rightarrow \infty} N^{d-2} \PROB{x_1}{ \tau_{x_2} < \tau_N \wedge \min_{j \notin \{1,2 \} } \tau_{x_j} }
& =
\lim_{N \rightarrow \infty} N^{d-2} \PROB{x_1}{ \tau_{x_2} < \tau_N } \\
& = \frac{a_d}{g} \left( \abs{y_1 - y_2}^{2-d} - q(y_1,y_2) \right).
\end{align*}
We deduce \eqref{eq:lem number excursions E asymptotic} by \eqref{eq: dim 3 proof lemma number}.
\end{proof}

We now prove Lemma \ref{lem: dim 3 two close points are not both thick}.

\begin{proof}[Proof of Lemma \ref{lem: dim 3 two close points are not both thick}]

Let $x \neq y \in V_N$ and let us denote
\[
p_{xy} := \PROB{x}{\tau_y<\infty} = \PROB{y}{\tau_x<\infty}
\mathrm{~and~}
\theta_{xy} = \EXPECT{x}{\ell_x^{\tau_y}} = \EXPECT{y}{\ell_y^{\tau_x}}.
\]
By decomposing the walk along the different excursions between $x$ and $y$, by Lemma \ref{lem:crucial local times hitting times} we see that starting from $x$ the joint law of $\left( \ell_{x}^\infty, \ell_{y}^\infty \right)$ can be stochastically dominated by:
\[
\left( \ell_x^\infty, \ell_y^\infty \right)
\preceq
\left( \sum_{j=1}^{A} \ell_{x,j} , \sum_{j=1}^{A} \ell_{y,j} \right)
\]
where $A$ is a geometric random variable with failure probability
\[
\left( p_{xy}\right)^2 = \PROB{x}{\exists 0 < s < t, Y_s = y, Y_t = x}
\]
and $\ell_{x,j}, \ell_{y,j}, j \geq 1,$ are i.i.d. exponential variables with mean $\theta_{xy}$ independent from $A$. $A$ is the number of round trips between $x$ and $y$ and $\ell_{x,j}$ is the time spent in $x$ during the $j$-th round trip.
Let us mention that it is not an exact equality in distribution but only a stochastic domination. Indeed, we exactly have: starting from $x$,
\begin{equation}\label{eq:proof decomposition l_x}
\ell_x^\infty \overset{\mathrm{(d)}}{=} \sum_{j=1}^{A} \ell_{x,j},
\end{equation}
but the number of $\ell_{y,j}$'s we have to sum up is $A$ (resp. $A-1$) if the last visited vertex is $y$ (resp. $x$).
However this stochastic domination is sufficient for our purposes.

Let $p \geq 0$. For all $i = 1 \dots p+1$ we stochastically dominate as above $\left( \ell_x^{\infty,i}, \ell_y^{\infty,i} \right)$ by variables with a superscript $i$ and we have
\begin{align*}
& \Prob{ \sum_{i=1}^{p+1} \ell_x^{\infty,i} \geq 2g\log N + gt , \sum_{i=1}^{p+1} \ell_y^{\infty,i} \geq 2g\log N + gt } \\
& \leq \Prob{ \sum_{i=1}^{p+1} \sum_{j=1}^{A^i} \ell_{x,j}^i \geq 2g \log N + gt, \sum_{i=1}^{p+1} \sum_{j=1}^{A^i} \ell_{y,j}^i \geq 2g \log N + gt}.
\end{align*}
Conditioned on the value of $\sum_{i=1}^{p+1} A^i$, the variables $\sum_{i=1}^{p+1} \sum_{j=1}^{A^i} \ell_{x,j}^i$ and $\sum_{i=1}^{p+1} \sum_{j=1}^{A^i} \ell_{y,j}^i$ are two independent Gamma variables. We can thus use the claim \eqref{eq:lem Gamma} of Lemma \ref{lem: dim 3 gamma laws} and

\begin{align}
& \Prob{ \sum_{i=1}^{p+1} \ell_x^{\infty,i} \geq 2g\log N + gt , \sum_{i=1}^{p+1} \ell_y^{\infty,i} \geq 2g\log N + gt } \nonumber \\
& \leq
N^{-4g/\theta_{xy}} e^{-2t}
\sum_{n = 0}^\infty \Prob{\sum_{i=1}^{p+1} A_i = n+p+1}
\sum_{q=0}^{2(n+p)} \frac{1}{q!} \left( 4 \frac{g}{\theta_{xy}} \log N \right)^q \nonumber \\
& = N^{-4g/\theta_{xy}} e^{-2t} \left( 1 - p_{xy}^2 \right)^{p+1} \sum_{n=0}^\infty p_{xy}^{2n} \binom{n+p}{p} \sum_{q=0}^{2(n+p)} \frac{1}{q!} \left( 4 \frac{g}{\theta_{xy}} \log N \right)^q \nonumber \\
& \leq C(p,t) N^{-4g/\theta_{xy}}
\sum_{q = 0}^\infty \frac{1}{q!} \left( 4 \frac{g}{\theta_{xy}} \log N \right)^q \sum_{n \geq \left( \lceil q/2 \rceil - p \right)_+ } (n+p) \dots (n+1) p_{xy}^{2n}. \label{eq: dim 3 lemma two close points}
\end{align}

We are going to bound from above the last sum indexed by $n$.
Let us first notice that $p_{xy}$ and $\theta_{xy}$ are linked by a simple formula. Indeed, \eqref{eq:proof decomposition l_x} implies that $\EXPECT{x}{\ell_x^\infty} = \Expect{A} \Expect{\ell_{x,1}}$, meaning that $g = \theta_{xy} / \left( 1 - p_{xy}^2 \right)$.
Then
\[
\inf_{x\neq y} g(1-p_{xy})/\theta_{xy} = \inf_{x \neq y} 1/(1+p_{xy}) > 1/2
\]
so we can find $\lambda >1$ such that $\inf_{x \neq y } g(1-\lambda p_{xy})/\theta_{xy} > 1/2.$
If the index $q$ in the equation \eqref{eq: dim 3 lemma two close points} is large enough, say $q \geq q_0(p)$, then for all $n \geq \lceil q/2 \rceil - p$ we have $2 \log (\lambda) n \geq p \log (n+p)$ and we can bound
\begin{align*}
\sum_{n \geq \left( \lceil q/2 \rceil - p \right)_+ }
(n+p) \dots (n+1) p_{xy}^{2n}
& \leq
\sum_{n \geq \lceil q/2 \rceil - p }
(n+p)^p p_{xy}^{2n} \\
& \leq
\sum_{n \geq \lceil q/2 \rceil - p }
\left( \lambda p_{xy}\right)^{2n}
\leq
C(p) \left( \lambda p_{xy} \right)^q.
\end{align*}
If $q < q_0(p)$, we bound the sum indexed by $n$ by some constant depending on $p$. Overall, coming back to the equation \eqref{eq: dim 3 lemma two close points}, we can further bound from above the probability we are interested in by:
\begin{align*}
C'(p,t)N^{-4g/\theta_{xy}} \left( (\log N)^{q_0(p)-1} + \sum_{q = q_0(p)}^\infty \frac{1}{q!} \left( 4 \frac{g}{\theta_{xy}} \lambda p_{xy} \log N \right) \right)
\leq C''(p,t) N^{-4 \frac{g(1-\lambda p_{xy})}{\theta_{xy}}}.
\end{align*}
We have chosen $\lambda$ to make sure that the previous exponent is smaller than $-2$ which is exactly what was required.
\end{proof}

We now state and prove elementary Lemma \ref{lem: dim 3 gamma laws} (recall the definition of $f(k \to q)$ in \eqref{eq:dim3 def f}).

\begin{lemma}\label{lem: dim 3 gamma laws}
1. Poisson distribution: For $\lambda > 0$, consider $P(\lambda)$ a Poisson random variable with parameter $\lambda$. Then for all $k \geq 1$,
\begin{equation}\label{eq:lem Poisson}
\Expect{P(\lambda)^k} = \sum_{q=1}^k f(k \to q) \lambda^q.
\end{equation}

2. Gamma distribution:
For $k,p \geq 1$ and $\theta>0$, consider $\Gamma_1(p,\theta), \dots, \Gamma_k(p,\theta)$ $k$ i.i.d. Gamma random variables with shape parameter $p$ and scale parameter $\theta$, which have the law of the sum of $p$ independent exponential variables with mean $\theta$. Then for all $T >0$,
\begin{equation}\label{eq:lem Gamma}
\Prob{ \forall i=1 \dots k, \Gamma_i(p,\theta) \geq T }
\leq 
e^{-k \frac{T}{\theta}} \sum_{q=0}^{k(p-1)} \left( k \frac{T}{\theta} \right)^q / (q!).
\end{equation}
\end{lemma}

\begin{proof}[Proof of Lemma \ref{lem: dim 3 gamma laws}]
1. Poisson distribution: The moment generating function of $P(\lambda)$ is given by: for all $u \in \R$
\begin{align*}
\Expect{e^{uP(\lambda)}}
& = \exp( \lambda(e^u - 1) )
= \sum_{q=0}^\infty \frac{\lambda^q}{q!} (e^u -1)^q
= \sum_{q=0}^\infty \frac{\lambda^q}{q!} \sum_{i=1}^q \binom{q}{i} (-1)^{q-i} e^{iu} \\
& = \sum_{q=0}^\infty \frac{\lambda^q}{q!} \sum_{i=1}^q \binom{q}{i} (-1)^{q-i} \sum_{k=0}^\infty i^k \frac{u^k}{k!} = \sum_{k=0}^\infty \frac{u^k}{k!} \sum_{q=0}^k \lambda^q f(k \to q)
\end{align*}
where $f$ is defined in \eqref{eq:dim3 def f}. This proves \eqref{eq:lem Poisson}.

2. Gamma distribution: The probability we are interested in is equal to
\begin{align*}
\Prob{ \Gamma_1(p,\theta) \geq T }^k
= e^{-k \frac{T}{\theta}} \left( \sum_{q=0}^{p-1} \left( \frac{T}{\theta} \right)^q / q! \right)^k
= e^{-k \frac{T}{\theta}} \sum_{q = 0}^{k(p-1)} \left( \frac{T}{\theta} \right)^q \sum_{\substack{0 \leq q_1, \dots, q_k \leq p-1 \\q_1 + \dots + q_k = q}} \frac{1}{q_1! \dots q_k!}.
\end{align*}
By looking at the power series of $x \mapsto (e^x)^k$ we find that
\[
\sum_{\substack{0 \leq q_1, \dots, q_k \leq p-1\\q_1 + \dots + q_k = q}} \frac{1}{q_1! \dots q_k!}
\leq
\sum_{\substack{q_1, \dots, q_k \geq 0\\q_1 + \dots + q_k = q}} \frac{1}{q_1! \dots q_k!}
=
\frac{k^q}{q!}
\]
which concludes the proof of \eqref{eq:lem Gamma}.
\end{proof}

We finish this paper by stating a lemma of measure theory. We include a proof for completeness and because we have not found any reference for this lemma.

\begin{lemma}\label{lem: decomposition function in simple functions}
Let $\phi : [-1,1]^d \times \R \to \R$ be a $\Cc^\infty$ function with compact support. Then there exists a sequence $(\phi_p)_{p \geq 1}$ of functions converging uniformly to $\phi$ such that for all $p \geq 1$,
\[
\phi_p = \sum_{i=1}^p a_i^{(p)} \mathbf{1}_{ A_i^{(p)} \times T_i^{(p)} }
\]
where $A_i^{(p)} \in \Bc ([-1,1]^d)$ with the Lebesgue measure of $\bar{A}^{(p)}_i \backslash (A_i^{(p)})^\circ$ vanishing, $T_i^{(p)} \in \Bc( \R )$ with $\inf T_i^{(p)} > - \infty$ and $a_i^{(p)} \in \C$.
\end{lemma}

\begin{proof}
Let $\eps>0$. As $\phi$ is $\Cc^\infty$ with compact support, the Fourier series of $\phi$ converges uniformly. We can thus find $K \geq 1$, $c_{k_x, k_t} \in \C, k_x \in \Z^d, k_t \in \Z$ and $t_0 \in \R$ such that the uniform norm of
\[
\phi - \sum_{ \substack{ k_x \in \Z^d, \norme{k_x} \leq K \\ k_t \in \Z, \abs{k_t} \leq K} } c_{k_x, k_t} e^{i k_x \cdot x} e^{i k_t \cdot t} \mathbf{1}_{(t_0, \infty)}
\]
is smaller than $\eps$. This procedure separates the variables $x$ and $t$. Now, writing $u_+$ and $u_-$ the positive and negative parts of a real $u$, we decompose
\[
e^{i k_x \cdot x} = \left( \cos(k_x \cdot x) \right)_+ - \left( \cos(k_x \cdot x) \right)_- + i \left( \sin(k_x \cdot x) \right)_+ - i \left( \sin(k_x \cdot x) \right)_-.
\]
Hence, we conclude this lemma by decomposing these four previous functions into sums of simple functions and we do the same thing for the variable $t$. We are going to detail this. In particular, we are going to explain how to ensure that the boundary of the Borel sets linked to the simple functions have zero Lebesgue measure. Let $\varphi : \R^d \rightarrow [0,\infty)$ be a continuous bounded function. We take $\xi > 0$ such that the Lebesgue measure of $\varphi^{-1} \left( \left\{ k2^{-p} - \xi, k \geq 1, p \geq 1 \right\} \right)$ vanishes. It is possible because the set of non suitable $\xi$'s is at most countable. Now we introduce 
\[
\psi_p := \sum_{k=0}^{p2^p} k 2^{-p} \mathbf{1}_{A_{p,k}}
\mathrm{~where~}
A_{p,k} = \varphi^{-1}\left( \left[ k 2^{-p} - \xi, (k+1)2^{-p} - \xi \right) \right).
\]
Thanks to our choice of $\xi$, the Lebesgue measure of $\bar{A}_{p,k} \backslash A_{p,k}^\circ$  vanishes. Also, since $\varphi + \xi$ is positive and bounded, $0 \leq (\varphi + \xi) - \psi_p \leq 2^{-p}$ for all $p$ large enough. We have thus uniformly approximated $\varphi$ by simple functions with Borel sets of the form we desired. This concludes the proof of the lemma.
\end{proof}

\paragraph*{Acknowledgements}

I am grateful to Nathanaël Berestycki for suggesting me these problems and for all the time he dedicated to me. I am also grateful to William Da Silva with whom I started this project for stimulating discussions, to Jay Rosen for bringing Nathanaël's attention to the Eisenbaum's isomorphism, to Nicolas Curien for a useful comment which substantially improved the statement of Theorem \ref{th: dim 3 convergence measures} and to Victor Dagard for proofreading a first version of the article. Finally, I would like to thank an anonymous referee for their careful reading of the paper.


\bibliographystyle{amsplain}


\begin{thebibliography}{10}

\bibitem{abe2015RayKnight}
Yoshihiro Abe, \emph{Maximum and minimum of local times for two-dimensional
  random walk}, Electron. Commun. Probab. \textbf{20} (2015), 14 pp.

\bibitem{AbeBiskup}
Yoshihiro Abe and Marek Biskup, \emph{Exceptional points of two-dimensional
  random walks at multiples of the cover time}, arXiv e-prints (2019).

\bibitem{BassRosen2007}
Richard Bass and Jay Rosen, \emph{Frequent points for random walks in two
  dimensions}, Electron. J. Probab. \textbf{12} (2007), 1--46.

\bibitem{bass1995}
Richard~F. Bass, \emph{Probabilistic techniques in analysis}, Probability and
  Its Applications, Springer-Verlag New York, 1995.

\bibitem{berestycki}
Nathanaël Berestycki, \emph{Introduction to the {G}aussian free field and
  {L}iouville quantum gravity}, Lecture notes. Available on the webpage of the
  author, 2016.

\bibitem{berestycki2017}
\bysame, \emph{An elementary approach to {G}aussian multiplicative chaos},
  Electron. Commun. Probab. \textbf{22} (2017), 12 pp.

\bibitem{BiskupLouidor}
Marek Biskup and Oren Louidor, \emph{On intermediate level sets of
  two-dimensional discrete {G}aussian free field}, Ann. Inst. H. Poincaré
  Probab. Statist. \textbf{55} (2019), no.~4, 1948--1987.

\bibitem{bolthausen2001}
Erwin Bolthausen, Jean-Dominique Deuschel, and Giambattista Giacomin,
  \emph{Entropic repulsion and the maximum of the two-dimensional harmonic
  crystal}, Ann. Probab. \textbf{29} (2001), no.~4, 1670--1692.

\bibitem{Chelkak_Smirnov2011}
Dmitry Chelkak and Stanislav Smirnov, \emph{Discrete complex analysis on
  isoradial graphs}, Advances in Mathematics \textbf{228} (2011), no.~3, 1590
  -- 1630.

\bibitem{chiarini2015}
Alberto Chiarini, Alessandra Cipriani, and Rajat Hazra, \emph{A note on the
  extremal process of the supercritical {G}aussian free field}, Electron.
  Commun. Probab. \textbf{20} (2015), 10 pp.

\bibitem{CFR_2005}
Endre Cs{\'a}ki, Ant{\'o}nia F{\"o}ldes, and P{\'a}l R{\'e}v{\'e}sz,
  \emph{Maximal local time of a d-dimensional simple random walk on subsets},
  Journal of Theoretical Probability \textbf{18} (2005), no.~3, 687--717.

\bibitem{CFRbis_2007}
\bysame, \emph{On the behavior of random walk around heavy points}, Journal of
  Theoretical Probability \textbf{20} (2007), no.~4, 1041--1057.

\bibitem{CFR_2007}
\bysame, \emph{On the local times of transient random walks}, Acta Applicandae
  Mathematicae \textbf{96} (2007), no.~1, 147--158.

\bibitem{CFR_2006}
Endre Csáki, Antónia Földes, and Pál Révész, \emph{Heavy points of a
  d-dimensional simple random walk}, Statistics \& Probability Letters
  \textbf{76} (2006), no.~1, 45 -- 57.

\bibitem{CFRter_2007}
Endre Csáki, Antónia Földes, and Pál Révész, \emph{Joint asymptotic
  behavior of local and occupation times of random walk in higher dimension},
  Studia Scientiarum Mathematicarum Hungarica \textbf{44} (2007), no.~4,
  535--563.

\bibitem{CFRRS_2005}
Endre Csáki, Antónia Földes, Pál Révész, Jay Rosen, and Zhan Shi,
  \emph{Frequently visited sets for random walks}, Stochastic Processes and
  their Applications \textbf{115} (2005), no.~9, 1503 -- 1517.

\bibitem{daviaud2006}
Olivier Daviaud, \emph{Extremes of the discrete two-dimensional gaussian free
  field}, Ann. Probab. \textbf{34} (2006), no.~3, 962--986.

\bibitem{DPRZ_2000_spatial}
Amir Dembo, Yuval Peres, Jay Rosen, and Ofer Zeitouni, \emph{Thick points for
  spatial brownian motion: multifractal analysis of occupation measure}, Ann.
  Probab. \textbf{28} (2000), no.~1, 1--35.

\bibitem{dembo2001}
\bysame, \emph{Thick points for planar {B}rownian motion and the
  {E}rdős-{T}aylor conjecture on random walk}, Acta Math. \textbf{186} (2001),
  no.~2, 239--270.

\bibitem{CoverBlanketTimes}
J.~{Ding}, J.~R. {Lee}, and Y.~{Peres}, \emph{{Cover times, blanket times, and
  majorizing measures}}, Ann. of Math. \textbf{175} (2012), 1409--1471.

\bibitem{erdos_taylor1960}
Paul Erd\H{o}s and Samuel~James Taylor, \emph{Some problems concerning the
  structure of random walk paths}, Acta Math. Acad. Sci. Hungar. \textbf{11}
  (1960), 137--162.

\bibitem{HuMillerPeres2010}
Xiaoyu Hu, Jason Miller, and Yuval Peres, \emph{Thick points of the {G}aussian
  free field}, Ann. Probab. \textbf{38} (2010), no.~2, 896--926.

\bibitem{jegoGMC}
Antoine Jego, \emph{Planar {B}rownian motion and {G}aussian multiplicative
  chaos}, Ann. Probab. (2018+), (to appear).

\bibitem{jegoRW}
\bysame, \emph{Characterisation of planar {B}rownian multiplicative chaos},
  ArXiv e-prints (2019).

\bibitem{kahane}
Jean-Pierre Kahane, \emph{Sur le chaos multiplicatif}, Ann. Sci. Math. Québec
  \textbf{9} (1985), no.~2, 105--150.

\bibitem{Kenyon2002}
Richard Kenyon, \emph{The {L}aplacian and {D}irac operators on critical planar
  graphs}, Inventiones mathematicae \textbf{150} (2002), no.~2, 409--439.

\bibitem{lawler1996intersections}
Gregory~F. Lawler, \emph{Intersections of random walks}, Probability and its
  Applications, Birkh{\"a}user Boston, 1996.

\bibitem{lawler_limic_2010}
Gregory~F. Lawler and Vlada Limic, \emph{Random walk: A modern introduction},
  Cambridge Studies in Advanced Mathematics, Cambridge University Press, 2010.

\bibitem{Revesz_2004}
P{\'a}l R{\'e}v{\'e}sz, \emph{The maximum of the local time of a transient
  random walk}, Studia Scientiarum Mathematicarum Hungarica \textbf{41} (2004),
  no.~4, 379--390.

\bibitem{rosen2006}
Jay Rosen, \emph{{A random walk proof of the Erd\H{o}s-Taylor conjecture}},
  Periodica Mathematica Hungarica \textbf{50} (2005), no.~1, 223--245.

\bibitem{rosen_2014}
Jay Rosen, \emph{Lectures on isomorphism theorems}, ArXiv e-prints (2014).

\bibitem{zeitouni_notes}
Ofer Zeitouni, \emph{Branching random walks and {G}aussian fields},  (2012),
  Lecture notes. Available on the webpage of the author.

\bibitem{Zhai2014}
Alex Zhai, \emph{Exponential concentration of cover times}, ArXiv e-prints
  (2014).

\end{thebibliography}


\providecommand{\bysame}{\leavevmode\hbox to3em{\hrulefill}\thinspace}
\providecommand{\MR}{\relax\ifhmode\unskip\space\fi MR }
\providecommand{\MRhref}[2]{%
  \href{http://www.ams.org/mathscinet-getitem?mr=#1}{#2}
}
\providecommand{\href}[2]{#2}


\end{document}